\documentclass[english]{article}
\usepackage[T1]{fontenc}
\usepackage[latin9]{inputenc}
\usepackage{color}
\usepackage{refstyle}
\usepackage{mathrsfs}
\usepackage{enumitem}
\usepackage{amsmath}
\usepackage{amsthm}
\usepackage{amssymb}
\usepackage{setspace}
\makeatletter

\AtBeginDocument{\providecommand\secref[1]{\ref{sec:#1}}}
\AtBeginDocument{\providecommand\propref[1]{\ref{prop:#1}}}
\AtBeginDocument{\providecommand\lemref[1]{\ref{lem:#1}}}
\AtBeginDocument{\providecommand\corref[1]{\ref{cor:#1}}}
\AtBeginDocument{\providecommand\defref[1]{\ref{def:#1}}}
\AtBeginDocument{\providecommand\exaref[1]{\ref{exa:#1}}}
\RS@ifundefined{subsecref}
  {\newref{subsec}{name = \RSsectxt}}
  {}
\RS@ifundefined{thmref}
  {\def\RSthmtxt{theorem~}\newref{thm}{name = \RSthmtxt}}
  {}
\RS@ifundefined{lemref}
  {\def\RSlemtxt{lemma~}\newref{lem}{name = \RSlemtxt}}
  {}

\theoremstyle{plain}
\newtheorem{thm}{\protect\theoremname}[section]
  \theoremstyle{plain}
  \newtheorem{lem}[thm]{\protect\lemmaname}
  \theoremstyle{plain}
  \newtheorem{prop}[thm]{\protect\propositionname}
  \theoremstyle{plain}
  \newtheorem{cor}[thm]{\protect\corollaryname}
  \theoremstyle{definition}
  \newtheorem{defn}[thm]{\protect\definitionname}
  \theoremstyle{remark}
  \newtheorem{rem}[thm]{\protect\remarkname}
  \theoremstyle{remark}
  \newtheorem*{rem*}{\protect\remarkname}
  \theoremstyle{definition}
  \newtheorem{example}[thm]{\protect\examplename}
 \newlist{casenv}{enumerate}{4}
 \setlist[casenv]{leftmargin=*,align=left,widest={iiii}}
 \setlist[casenv,1]{label={{\itshape\ \casename} \arabic*.},ref=\arabic*}
 \setlist[casenv,2]{label={{\itshape\ \casename} \roman*.},ref=\roman*}
 \setlist[casenv,3]{label={{\itshape\ \casename\ \alph*.}},ref=\alph*}
 \setlist[casenv,4]{label={{\itshape\ \casename} \arabic*.},ref=\arabic*}

\@ifundefined{date}{}{\date{}}
\AtBeginDocument{}
\AtBeginDocument{\providecommand\defref[1]{\ref{def:#1}}}
\AtBeginDocument{}
\AtBeginDocument{}
\AtBeginDocument{\providecommand\propref[1]{\ref{prop:#1}}}
\RS@ifundefined{propref}{\newref{prop}{name =Proposition~,names = Propositions~}}{}
\AtBeginDocument{\providecommand\corref[1]{\ref{cor:#1}}}
\RS@ifundefined{corref}{\newref{cor}{name =Corollary~,names = Corollaries~}}{}
\RS@ifundefined{defref}{\newref{def}{name =Definition~,names = Definitions~}}{}
\RS@ifundefined{exaref}{\newref{exa}{name =Example~,names = Examples~}}{}
\AtBeginDocument{\renewcommand\secref[1]{Section~\ref{sec:#1}}}

\usepackage{tikz}
\usepackage{tocbibind}
\usepackage{youngtab}
\usepackage{authblk}
\usetikzlibrary{matrix}

\DeclareMathOperator{\Rad}{Rad}

\DeclareMathOperator{\PT}{\mathcal{PT}}
\DeclareMathOperator{\E}{\mathcal{E}}
\DeclareMathOperator{\PO}{\mathcal{PO}}
\DeclareMathOperator{\PF}{\mathcal{PF}}
\DeclareMathOperator{\PC}{\mathcal{PC}}
\DeclareMathOperator{\EO}{\mathcal{EO}}
\DeclareMathOperator{\EC}{\mathcal{EC}}
\DeclareMathOperator{\IO}{\mathcal{IO}}
\DeclareMathOperator{\IF}{\mathcal{IF}}
\DeclareMathOperator{\IC}{\mathcal{IC}}
\DeclareMathOperator{\C}{\mathcal{C}}
\DeclareMathOperator{\EF}{\mathcal{EF}}
\DeclareMathOperator{\SEO}{\mathcal{SEO}}
\DeclareMathOperator{\IS}{\mathcal{IS}}

\DeclareMathOperator{\Lc}{\mathcal{L}}
\DeclareMathOperator{\Hc}{\mathcal{H}}

\DeclareMathOperator{\dom}{\mathsf{dom}}

\DeclareMathOperator{\op}{op}

\DeclareMathOperator{\db}{\mathbf{d}}
\DeclareMathOperator{\rb}{\mathbf{r}}
\DeclareMathOperator{\SP}{span}
\DeclareMathOperator{\VS}{\mathbf{VS}}
\DeclareMathOperator{\df}{defect}

\AtBeginDocument{\renewcommand\Subsecref[1]{Section~\ref{subsec:#1}}}

\def\RSlemtxt{Lemma~}
\def\RSthmtxt{Theorem~}

\usepackage{babel}
\providecommand{\corollaryname}{Corollary}
  \providecommand{\definitionname}{Definition}
  \providecommand{\examplename}{Example}
  \providecommand{\lemmaname}{Lemma}
  \providecommand{\propositionname}{Proposition}
  \providecommand{\remarkname}{Remark}
 \providecommand{\casename}{Case}
\providecommand{\theoremname}{Theorem}

\makeatother

\usepackage{babel}
  \providecommand{\corollaryname}{Corollary}
  \providecommand{\definitionname}{Definition}
  \providecommand{\examplename}{Example}
  \providecommand{\lemmaname}{Lemma}
  \providecommand{\propositionname}{Proposition}
  \providecommand{\remarkname}{Remark}
 \providecommand{\casename}{Case}
\providecommand{\theoremname}{Theorem}

\begin{document}

\title{Representation theory of order-related monoids of partial functions
as locally trivial category algebras}

\author{Itamar Stein\\
 Mathematics Unit\\
Shamoon College of Engineering\\
Israel\\
Steinita@gmail.com}
\maketitle
\begin{abstract}
In this paper we study the representation theory of three monoids
of partial functions on an $n$-set. The monoid of all order-preserving
functions (i.e., functions satisfying $f(x)\leq f(y)$ if $x\leq y$)
the monoid of all order-decreasing functions (i.e. functions satisfying
$f(x)\leq x$) and their intersection (also known as the partial Catalan
monoid). We use an isomorphism between the algebras of these monoids
and the algebras of some corresponding locally trivial categories.
We obtain an explicit description of a quiver presentation for each
algebra. Moreover, we describe other invariants such as the Cartan
matrix and the Loewy length.
\end{abstract}

\section{Introduction}

Given a finite monoid $M$, it is of interest to study its algebra
$\Bbbk M$ over some field $\Bbbk$. Monoids with natural combinatorial
structure are clearly of major interest. Denote by $\PT_{n}$ the
monoid of all partial functions on an $n$-element set $\{1,\ldots,n\}$.
A partial function $f$ is called \emph{order-preserving} if $x\leq y$
implies that $f(x)\leq f(y)$ for all $x,y$ in the domain of $f$.
$f$ is called \emph{order-decreasing} if $f(x)\leq x$ for every
$x$ in the domain of $f$. We denote by $\PO_{n}$ the submonoid
of $\PT_{n}$ consisting of all order-preserving partial functions
and by $\PF_{n}$ the submonoid of all order-decreasing partial functions.
The intersection $\PC_{n}=\PO_{n}\cap\PF_{n}$ is called the partial
Catalan monoid. These monoids are well-studied. For instance, see
\cite[Chapter 14]{Ganyushkin2009b} and references therein. Denote
by $\E_{n}$ the finite category whose objects are all subsets of
$\{1,\ldots,n\}$ and given two subsets $A,B\subseteq\{1,\ldots,n\}$
the hom-set $\E_{n}(A,B)$ consists of all onto functions with domain
$A$ and range $B$. In \cite[Section 5]{Stein2016} the author proved
that each one of the monoid algebras $\Bbbk\PO_{n}$, $\Bbbk\PF_{n}$
and $\Bbbk\PC_{n}$ is isomorphic to a category algebra of some corresponding
subcategory of $\E_{n}$. These are actually examples of an isomorphism
that holds for a larger class of semigroups (for details, see \cite{Stein2017erratum,Stein2017,Wang2017}).
In \cite[Section 5]{Stein2016} this isomorphism was used in order
to describe the ordinary quiver of these algebras. In this paper we
continue this study and obtain a description of other invariants of
these algebras. Our main result is a description of a quiver presentation
for these algebras. We remark that currently there are relatively
few (non-trivial) cases of monoid algebras for which a description
of a quiver presentation is known. Some of these known cases can be
found in \cite{Margolis2015,Ringel2000,Saliola2009}. A \emph{quiver
presentation} of an algebra $A$ is a standard way to ``present''
an algebra in the theory of associative algebras. It consists of a
(unique) ``generating'' graph $Q$ called the \emph{(ordinary)}
\emph{quiver} of $A$ and a (non-unique) set $R$ of ``relations''
between paths in $Q$. A tuple $(Q,R)$ presents an algebra $B$ which
does not need to be isomorphic to $A$ but it is Morita equivalent
to $A$ (i.e., the module categories of $A$ and $B$ are equivalent).
Therefore $A$ and $B$ share most of the important invariants of
representation theory. Another important invariant that we will study
is the Cartan matrix of an algebra, which gives the Jordan-Hölder
decomposition of every (indecomposable) projective module into simple
factors. We will also find the Loewy length of these algebras and
their decomposition into a direct product of connected algebras.

The key fact in studying the algebras $\Bbbk\PO_{n}$, $\Bbbk\PF_{n}$
and $\Bbbk\PC_{n}$ is that each of their corresponding categories
is \emph{locally trivial, }i.e., the only endomorphisms are the identity
maps of the category. We remark that locally trivial categories were
used in monoid theory by Tilson in \cite{Tilson1987}. In \secref{AlgebrasOfLocallyTrivialCat}
we study algebras of locally trivial categories in general. The main
observation is that the description of many invariants of the category
algebra can be reduced to some question about the category itself.
Finding a quiver presentation for the algebra, can be reduced to finding
a presentation for the category itself. Therefore we can reduce a
representation theoretic problem into a combinatorial problem. Likewise
the Cartan matrix has an immediate interpretation via hom-sets of
the category. The connectedness of the algebra is reduced to connectedness
of the category as a graph. In \secref{RepresentationTheoryOfMonoidsOfPartialFunctions}
we use these observation to study $\Bbbk\PO_{n}$, $\Bbbk\PF_{n}$
and $\Bbbk\PC_{n}$ and obtain an explicit descriptions of their invariants
or at least an explicit way to compute them.

\section{Preliminaries}

\subsection{Categories and $\mathbb{\Bbbk}$-linear categories}

\paragraph{Graphs}

Let $Q$ be a (finite, directed) graph. We denote by $Q^{0}$ its
set of vertices (or \emph{objects}) and by $Q^{1}$ its set of edges
(or \emph{morphisms}). We allow $Q$ to have more than one morphisms
between two objects. We denote by $\db$ and $\rb$ the domain and
range functions
\[
\db,\rb:Q^{1}\to Q^{0}
\]
which associate to every edge $m\in Q^{1}$ its \emph{domain $\db(m)$
}and \emph{range $\rb(m)$ }respectively\emph{.} The set of edges
with domain $a$ and range $b$ (also called an \emph{hom-set}) is
denoted $Q(a,b)$.

\paragraph{Categories and their presentations}

Let $E$ be a finite category. Since any category has an underlying
graph, the above notations hold also for $E$.\emph{ }A \emph{relation}
$R$ on a category $E$ is a relation on the set of morphisms $E^{1}$
such that $mRm^{\prime}$ implies that $\db(m)=\db(m^{\prime})$ and
$\rb(m)=\rb(m^{\prime})$. A relation $\theta$ on a category $E$
is called a (category) congruence if $\theta$ is an equivalence relation
with the property that $m_{1}\theta m_{2}$ and $m_{1}^{\prime}\theta m_{2}^{\prime}$
implies $m_{1}^{\prime}m_{1}\theta m_{2}^{\prime}m_{2}$ whenever
$\rb(m_{1})=\db(m_{1}^{\prime})$ and $\rb(m_{2})=\db(m_{2}^{\prime})$.
The quotient category $E/\theta$ is then defined in a natural way.
The objects of $E$ and $E/\theta$ are identical and the morphisms
of $E/\theta$ are the equivalence classes of $\theta$. The fact
that intersection of congruences is a congruence implies that any
relation $R$ on $E$ generates some congruence $\theta$, that is,
there exists some congruence $\theta$ which is the minimal congruence
on $E$ containing $R$. We will denote this congruence by $\theta_{R}$.
For every graph $Q$, denote by $Q^{\ast}$ the free category generated
by the graph $Q$. The object set of $Q^{\ast}$ and $Q$ are identical
but the morphisms of $Q^{\ast}$ are the paths in $Q$ (including
one empty path for every object). The composition in $Q^{\ast}$ is
concatenation (from right to left) of paths and the empty paths are
the identity elements. We say that a subgraph $Q$ of $E$ \emph{generates}
$E$ if $Q^{0}=E^{0}$ and every non-identity morphism in $E$ can
be written as a composition of morphisms from $Q$. In this case,
one can define a congruence $\theta$ on $Q^{\ast}$ relating paths
that represent the same morphism in $E$ and obviously $Q^{\ast}/\theta\simeq E$.
Equivalently one can define a projection functor $\pi:Q^{\ast}\to E$
which is identity on objects and sends every path to the morphism
in $E^{1}$ it represent. Then $\pi(p_{1})=\pi(p_{2})$ for two paths
$p_{1},p_{2}\in(Q^{\ast})^{1}$ if and only if $p_{1}\theta p_{2}$.
If $R$ is a relation on $Q^{\ast}$ that generates $\theta$ (i.e.,
$\theta=\theta_{R}$) the convention is to call the morphisms of $Q$
\emph{generators} and the elements of $R$ \emph{relations}. We also
say that $E$ is \emph{presented} by the generators $Q$ and relations
$R$ and that $(Q,R)$ is a \emph{presentation} of $E$. Note that
the set of relations defining a congruence is not unique even if $Q$
is fixed. Note also that if $E$ is a monoid (viewed as a category
with one object) then this definition reduces to the usual definition
of a presentation of a monoid by generators and relations. It will
be useful to view a category presentation also as a coequalizer. Consider
a graph $Q$ and a relation $R=\{(m_{i},m_{i}^{\prime})\mid i\in I\}$
on $Q^{\ast}$. Denote by ${\bf 1}$ the graph with two objects and
one morphism connecting them, i.e, a graph that looks like $\ast\to\ast$.
It is clear that any functor $F:{\bf 1}^{\ast}\to Q^{\ast}$ corresponds
to choosing a morphism of $Q^{\ast}$. Denote by ${\bf I}$ a disjoint
union of $|I|$ copies of ${\bf 1}$. A functor $F:{\bf I^{\ast}\to Q^{\ast}}$
can be defined by associating each index $i\in I$ with a morphism
$F(i)$ of $Q^{\ast}$. Now define two functors $M,M^{\prime}:{\bf I}^{\ast}\to Q^{\ast}$
by $M(i)=m_{i}$ and $M^{\prime}(i)=m_{i}^{\prime}$. The category
$E$ which $(Q,R)$ presents is easily seen to be the coequalizer
of the diagram 
\[
{\bf I}^{\ast}\begin{array}{c}
\overset{M}{\to}\\
\underset{M^{\prime}}{\to}
\end{array}Q^{\ast}
\]
in the category of (small) categories.

\paragraph{$\Bbbk$-linear categories and their presentations}

A $\mathbb{\Bbbk}$-linear category, is a category $L$ enriched over
the category of $\mathbb{\Bbbk}$-vector spaces $\VS_{\mathbb{\Bbbk}}$.
This means that every hom-set of $L$ is a $\mathbb{\Bbbk}$-vector
space and the composition of morphisms is a bilinear map with respect
to the vector space operations. A \emph{functor} of $\mathbb{\Bbbk}$-linear
categories is a category functor which is also a linear transformation
when restricted to any hom-set. For any category $E$, we associate
a $\mathbb{\Bbbk}$-linear category $L_{\Bbbk}[E]$, called the \emph{linearization}
of $E$, defined in the following way. The objects of $L_{\Bbbk}[E]$
and $E$ are identical, and every hom-set $L_{\Bbbk}[E](a,b)$ is
the $\Bbbk$-vector space with basis $E(a,b)$. The composition of
morphisms in $L_{\Bbbk}[E]$ is defined naturally in the only way
that extends the composition of $E$ and forms a bilinear map. It
is easy to see that $L_{\Bbbk}$ is actually a functor from the category
of (small) categories to the category of (small) $\Bbbk$-linear categories.
It is not difficult to check that it is the left adjoint of the natural
forgetful functor from $\mathbb{\Bbbk}$-linear categories to categories.
Let $L$ be a $\mathbb{\Bbbk}$-linear category. A relation $\rho$
on $L^{1}$ is called a ($\mathbb{\Bbbk}$-linear category) congruence
if $\rho$ is a category congruence and also a vector space congruence
on every hom-set $L(a,b)$. By a vector space congruence on a vector
space $V$ we mean an equivalence relation $\rho$ on $V$ such that
$u_{1}\rho v_{1}$ and $u_{2}\rho v_{2}$ implies $(u_{1}+u_{2})\rho(v_{1}+v_{2})$
and $u\rho v$ implies $(cu)\rho(cv)$ for every $c\in\Bbbk$. The
quotient $\Bbbk$-linear category $L/\rho$ is then defined in the
natural way. For every relation $R$ on $L$ there exists a unique
minimal congruence on $L$ that contains $R$. We will denote this
congruence by $\rho_{R}$. Now we can define a presentation of $\mathbb{\Bbbk}$-linear
categories. Let $Q$ be a subgraph of $L$ such that $Q^{0}=L^{0}$.
The free $\mathbb{\Bbbk}$-linear category generated by $Q$ is the
category $L_{\Bbbk}[Q^{\ast}]$ of all linear combinations of paths.
If $R$ is a relation on $L_{\Bbbk}[Q^{\ast}]$ such that $L_{\Bbbk}[Q^{\ast}]/\rho_{R}\simeq L$,
then we say that $(Q,R)$ is a \emph{presentation} of $L$. Again,
$L$ can be viewed as a coequalizer. Assume $R=\{(m_{i},m_{i}^{\prime})\mid i\in I\}$
and define ${\bf 1}$ and ${\bf I}$ as before. It is clear that a
functor $F:L_{\Bbbk}[{\bf I^{\ast}]\to}L_{\Bbbk}[Q^{\ast}]$ can be
defined by associating each index $i\in I$ with a morphism $F(i)$
of $L_{\Bbbk}[Q^{\ast}]$. Define two functors $M,M^{\prime}:L_{\Bbbk}[{\bf I}^{\ast}]\to L_{\Bbbk}[Q^{\ast}]$
by $M(i)=m_{i}$ and $M^{\prime}(i)=m_{i}^{\prime}$. Then the category
$L$ is the coequalizer of the following diagram.
\[
L_{\Bbbk}[{\bf I}^{\ast}]\begin{array}{c}
\overset{M}{\to}\\
\underset{M^{\prime}}{\to}
\end{array}L_{\Bbbk}[Q^{\ast}]
\]

Note that any category $E$ can be naturally regarded as a subcategory
of $L_{\Bbbk}[E]$. Therefore, if $Q$ is a subgraph of $E$ then
it is clearly a subgraph of $L_{\Bbbk}[E]$ as well. If $R$ is a
relation on $Q^{\ast}$ then it can also be regarded as a relation
on $L_{\Bbbk}[Q^{\ast}]$. This allows us to state the following simple
observation that will be useful in the sequel.
\begin{lem}
\label{lem:PresentationCatToLinearCat}Let $Q$ be a subgraph of $E$
and let $R$ be a relation on $Q^{\ast}$ such that $(Q,R)$ is a
category presentation for $E$. Then $(Q,R)$ is also a $\mathbb{\Bbbk}$-linear
category presentation for $L_{\Bbbk}[E]$.
\end{lem}
\begin{proof}
As described above, $E$ is the coequalizer of the following diagram
\[
{\bf I}^{\ast}\begin{array}{c}
\overset{M}{\to}\\
\underset{M^{\prime}}{\to}
\end{array}Q^{\ast}
\]
where ${\bf I}$, $M$ and $M^{\prime}$ are as defined above. Applying
the functor $L_{\Bbbk}$, we obtain the diagram 
\[
L_{\Bbbk}[{\bf I}^{\ast}]\begin{array}{c}
\overset{M}{\to}\\
\underset{M^{\prime}}{\to}
\end{array}L_{\Bbbk}[Q^{\ast}].
\]
Since $L_{\Bbbk}$ is a left adjoint it preserves coequalizers (\cite[Chapter V, Section 5]{MacLane1998}).
Hence $L_{\Bbbk}[E]$ is the coequalizer of the second diagram which
is precisely what we want to prove.
\end{proof}
More on categories and linear categories can be found in \cite{MacLane1998}.

\subsection{Algebras and representations}

A \emph{representation} of a $\Bbbk$-linear category $L$ is a functor
of $\Bbbk$-linear categories from $L$ to the category of all $\Bbbk$-vector
spaces $\VS_{\Bbbk}$. Recall that a $\mathbb{\Bbbk}$-algebra is
a $\mathbb{\Bbbk}$-linear category with one object. We will mainly
be interested in category algebras. For some (finite) category $E$,
the \emph{category algebra} $\mathbb{\Bbbk}E$ is defined in the following
way. It is a vector space over $\mathbb{\Bbbk}$ with basis the morphisms
of $E$, that is, it consists of all formal linear combinations
\[
\{k_{1}m_{1}+\ldots+k_{n}m_{n}\mid k_{i}\in\mathbb{\Bbbk},\,m_{i}\in E^{1}\}.
\]
The multiplication in $\mathbb{\Bbbk}E$ is the linear extension of
the following:
\[
m^{\prime}\cdot m=\begin{cases}
m^{\prime}m & \db(m^{\prime})=\rb(m)\\
0 & \text{otherwise.}
\end{cases}
\]
Since a monoid $M$ is a category with one object, this definition
also gives a definition for monoid algebras. In this case the monoid
algebra contains linear combinations of elements of the monoid with
the obvious multiplication. If $M$ has a zero element $0\in M$ then
$\Bbbk\{0\}$ is an ideal of $\Bbbk M$. In this special case we also
define $\mathbb{\Bbbk}_{0}M=\mathbb{\Bbbk}M/\mathbb{\Bbbk}\{0\}$. 

Let $A$ be some $\mathbb{\Bbbk}$-algebra. Unless stated otherwise,
we assume that algebras are finite dimensional. \textcolor{black}{Recall
that two idempotents }$e,f\in A$\textcolor{black}{{} are called orthogonal
if }$ef=fe=0$.\textcolor{black}{{} A non-zero idempotent} $e\in A$
is called \emph{primitive} if it is not a sum of two non-zero orthogonal
idempotents. This is equivalent to $eAe$ being a local algebra (i.e.,
an algebra with no non-trivial idempotents). A \emph{complete set
of primitive orthogonal idempotents }is a set of primitive, mutually
orthogonal idempotents $\{e_{1},\ldots,e_{r}\}$ whose sum is $1$.
Recall that the radical of $A$ is the minimal ideal such that $A/\Rad A$
is a semisimple algebra. $A$ is called \emph{split basic} if $A/\Rad A\simeq\Bbbk^{n}$,
i.e, the maximal semisimple quotient of $A$ is a direct product of
the base field. We can associate to every algebra $A$ a linear category,
denoted $\mathscr{L}(A)$, in the following way. The objects are in
one-to-one correspondence with a complete set of primitive idempotents.
The hom-set $\mathscr{L}(A)(e_{i},e_{j})$ is the set $e_{j}Ae_{i}$,
i.e., all the elements $a\in A$ such that $e_{j}ae_{i}=a$. Composition
of two morphisms is naturally defined as their product in the algebra
$A$. The category of all $A$-representations is equivalent to the
category of all $\mathscr{L}(A)$-representations (no matter which
complete set of primitive idempotents is chosen). There are some important
invariants of the algebra $A$ that can be described by the category
$\mathscr{L}(A)$ as we will see immediately. Let $A$ be a split
basic algebra.\textcolor{red}{{} }The \emph{(ordinary) quiver} of $A$
is a directed graph $Q$ defined in the following way: The set of
vertices of $Q$ is in one-to-one correspondence with $\{e_{1},\ldots,e_{n}\}$
and the edges (more often called \emph{arrows}) from $e_{k}$ to $e_{r}$
are in one to one correspondence with some basis of the vector space
$(e_{r}+\Rad^{2}A)(\Rad A/\Rad^{2}A)(e_{k}+\Rad^{2}A)$. It is well
known that this definition does not depend on the exact choice of
the primitive orthogonal idempotents. The quiver $Q$ of $A$ can
be identified with a certain subgraph of $\mathscr{L}(A)$. Actually,
there exists a $\Bbbk$-linear relation $R$ on $L_{\Bbbk}[Q^{\ast}]$
such that $(Q,R)$ is a presentation for $\mathscr{L}(A)$ (\cite[Theorem 1.9 on page 65]{Auslander1997}).
Such a pair $(Q,R)$ is called a \emph{quiver presentation }for the
algebra $A$. We briefly recall some other definitions related to
an algebra. An algebra $A$ is connected if $0,1$ are its only central
idempotents. If $A$ is not connected then it is a direct product
of connected algebras called the \emph{blocks} of $A$. It is well
known that the number of blocks of $A$ is the number of the connected
components of its quiver $Q$ (as a graph). \textcolor{black}{The
}\textcolor{black}{\emph{Cartan matrix }}\textcolor{black}{of }$A$\textcolor{black}{{}
is an }$n\times n$ matrix whose $(i,j)$ entry is $\dim_{\Bbbk}e_{i}Ae_{j}$.
The \emph{descending Loewy series} of an algebra $A$ is the decreasing
sequence of ideals 
\[
0\subsetneq\ldots\subsetneq\Rad^{2}A\subsetneq\Rad A\subsetneq A
\]
and the \emph{Loewy length} of $A$ is the minimal integer $n$ such
that $\Rad^{n}A=0$. More facts on representations of algebras and
proofs can be found in \cite{Assem2006,Auslander1997}.

\section{Algebras of locally trivial categories\label{sec:AlgebrasOfLocallyTrivialCat}}

Recall that an \emph{endomorphism }in a category\emph{ }is a morphism
whose domain and range are equal. A category $E$ is called \emph{locally
trivial }if the only endomorphisms of $E$ are the identity morphisms.
In this section we will be interested in invariants related to the
structure of the category algebra $\Bbbk E$ such as the quiver presentation
and the Cartan matrix. We show that the description of these invariants
can be reduced quite easily to some questions about the structure
of the category $E$ itself. All facts in this section appear in the
literature or folklore. However the proofs are usually very simple
and we will give some of them for the sake of completeness. Note that
a partial order (considered as a category) is a special case of a
locally trivial category. Therefore, the results in this section generalize
some well known-facts about incidence algebras (i.e. algebras of partial
orders). Moreover, some of the results in this section were proved
for the more general case of an EI-category, i.e., a category where
every endomorphism is an isomorphism.

We start with describing the Jacobson radical of a locally trivial
category algebra. The following proposition is a special case of \cite[Proposition 4.6]{Li2011}.
\begin{prop}
\label{prop:RadOfEICat}Let $E$ be a finite locally trivial category.
Then the Jacobson radical $\Rad\mathbb{\mathbb{\Bbbk}}E$ is spanned
by all the non-isomorphisms of $E$.
\end{prop}
For simplicity of notation, we will write $\mathcal{R}(E)=\Rad\mathbb{\Bbbk}E$.
\begin{cor}
\label{cor:RadToThek}Let $E$ be a finite and locally trivial category.
$\mathcal{R}^{k}(E)$ is spanned by all the morphisms that can be
written as a composition of $k$ non-isomorphisms.
\end{cor}
Most of the invariants we will discuss are preserved under Morita
equivalence. For studying invariants of this type one can use the
skeleton of the category instead of the category itself. Recall that
a category $E$ is called \emph{skeletal} if it has no distinct isomorphic
objects. The \emph{skeleton} of a category $E$ is the full subcategory
obtained from $E$ by choosing one object of every isomorphism class
(recall that $D$ is a full subcategory of $E$ if $D^{1}(a,b)=E^{1}(a,b)$
for every $a,b\in D^{0}$). It is clear that the skeleton of any category
is a skeletal category and it is unique up to isomorphism. Moreover,
it is clear that a category and its skeleton are equivalent categories.
Algebras of equivalent categories are Morita equivalent (see \cite[Proposition 2.2]{Webb2007})
and hence have the same quiver presentation, Cartan matrix etc. Note
that if $E$ is locally trivial then its skeleton is locally trivial
as well. Skeletal locally trivial categories were called \emph{deltas}
in \cite[Section 22]{Mitchell1972}.

One important observation is that the objects of a skeletal locally
trivial category are partially ordered in a natural way.
\begin{defn}
\label{def:EICatAreDirected}Let $E$ be a skeletal locally trivial
category. Define a relation $\leq_{E}$ on $E^{0}$ by $a\leq_{E}b$
if $E(a,b)\neq\varnothing$.
\end{defn}
\begin{lem}
\cite[Page 170]{Luck}\label{lem:PartialOrderOnLocallyTrivialCat}
$\leq_{E}$ is a partial order.
\end{lem}
From now on let $E$ be a finite skeletal and locally trivial category.
We denote the objects of $E$ by $\{e_{1},\ldots,e_{n}\}$ and identify
them with their corresponding identity morphism so we regard $e_{i}$
both as an object and as an identity morphism. We give some easy observations
on the algebra $\mathbb{\Bbbk}E$. The proof of the following lemma
is immediate and will be omitted.
\begin{lem}
\label{lem:PrimitiveIdempotentsCatTrivialMonoids}The set $\{e_{1},\ldots,e_{n}\}$
is a complete set of primitive orthogonal idempotents of $\mathbb{\Bbbk}E$.
\end{lem}
Another good property of a skeletal locally trivial category $E$
is that $\Bbbk E$ is a split basic algebra.
\begin{cor}
\label{cor:AlgebrasOfLocallyTrivialCatAreBasic}$\mathbb{\Bbbk}E$
is a split basic algebra.
\end{cor}
\begin{proof}
\propref{RadOfEICat} implies that $\mathbb{\Bbbk}E/\Rad\mathbb{\Bbbk}E$
is the algebra of the groupoid of isomorphisms of $E$. By \emph{groupoid
of isomorphisms} we mean the subcategory of $E$ with the same set
of objects but whose morphisms are the isomorphisms of $E$. Since
$E$ is skeletal and locally trivial, the only morphisms of this groupoid
are the identity morphisms $\{e_{1},\ldots,e_{n}\}$ so it is clear
that 
\[
\mathbb{\Bbbk}E/\Rad\mathbb{\Bbbk}E\simeq\mathbb{\Bbbk}^{n}
\]
as required. 
\end{proof}
Denote by $Q$ the quiver of $\mathbb{\Bbbk}E$. It will be sometimes
convenient to call it simply the quiver of $E$ (and likewise for
any other invariant of $\mathbb{\Bbbk}E$ we will discuss). The description
of the quiver of a locally trivial category is given in \cite[Theorem 6.14]{Margolis2012}
as a special case of a formula for the quiver of an EI-category (see
\cite[Theorem 4.7]{Li2011} or \cite[Theorem 6.13]{Margolis2012}).
However, for a locally trivial category $E$ it is quite straightforward
to obtain a description for the quiver so we will give it here for
the sake of completeness. The vertices of the quiver of a split basic
algebra are in one to one correspondence with a complete set of primitive
orthogonal idempotents so in our case they are corresponding to $\{e_{1},\ldots,e_{n}\}$
by \lemref{PrimitiveIdempotentsCatTrivialMonoids} and \corref{AlgebrasOfLocallyTrivialCatAreBasic}.
Therefore we can consider $E^{0}$ as the vertex set of $Q$. In order
to describe the arrows of the quiver we need some more notions. 
\begin{defn}
A morphism $m$ of a locally trivial category $E$ is called \emph{irreducible}
if it is a non-isomorphism but whenever $m=m^{\prime}m^{\prime\prime}$,
either $m^{\prime}$ is an isomorphism or $m^{\prime\prime}$ is an
isomorphism.
\end{defn}
\begin{rem}
The definition of an irreducible morphism in a general category can
be found in \cite[subsection 6.1]{Margolis2012}.
\end{rem}
It is easy to see that a morphism $m$ is irreducible if and only
if $m\in\mathcal{R}(E)\backslash\mathcal{R}^{2}(E)$ so the irreducible
morphisms of $E$ correspond to a basis of $\mathcal{R}(E)/\mathcal{R}^{2}(E)$.
 Concluding, we have the following observation:
\begin{lem}[{\cite[Theorem 6.14]{Margolis2012}}]
\label{lem:QuiverOfCatTrivialMonoids}Let $E$ be a finite skeletal
and locally trivial category. The quiver of $E$ is the subgraph of
$E$ having the same set of objects and the irreducible morphisms
as arrows.
\end{lem}
\lemref{QuiverOfCatTrivialMonoids} has an immediate application regarding
blocks. 
\begin{lem}
\label{lem:ConnectedComponentsOfLocallyTrivialCatAlg}Let $E$ be
a skeletal locally trivial category. Then the number of blocks of
$\mathbb{\Bbbk}E$ is the number of connected components of $E$.
\end{lem}
\begin{proof}
Recall that the number of blocks of the algebra $\mathbb{\Bbbk}E$
is the number of connected components of its quiver $Q$ (see \cite[Chapter II, Lemma 2.5]{Assem2006}).
By \lemref{QuiverOfCatTrivialMonoids} we can regard $Q$ as a subgraph
of $E$. It is obvious that two objects in the same connected component
of $Q$ are in the same connected component of $E$. On the other
direction, it is enough to prove that if there is a morphism $m\in E(e,e^{\prime})$
from $e$ to $e^{\prime}$ in $E$ then the objects $e$ and $e^{\prime}$
are in the same connected component of $Q$. Indeed, $m$ can be decomposed
into a composition of irreducible morphisms
\[
m=m_{k}\cdots m_{1}
\]

where $\db(m_{1})=e$ and $\rb(m_{k})=e^{\prime}$. Since $m_{i}\in Q^{1}$
for every $1\leq i\leq k$, we have that $e$ and $e^{\prime}$ are
in the same connected component of $Q$ as required.
\end{proof}
We now turn to describe a quiver presentation of $\mathbb{\Bbbk}E$
where $E$ is a skeletal locally trivial category. As we have already
seen, the primitive idempotents of $\mathbb{\Bbbk}E$ correspond to
the objects of $E$ and an element $x\in\mathbb{\Bbbk}E$ will satisfy
$e_{j}xe_{i}=x$ if and only if it is a linear combination of elements
from $E(e_{i},e_{j})$. Therefore, in this special case we have that
$\mathscr{L}[\mathbb{\Bbbk}E]=L_{\Bbbk}[E]$. In other words, the
$\mathbb{\Bbbk}$-linear category corresponding to the algebra $\Bbbk E$
is just the linearization of the category $E$. 

Therefore, \lemref{PresentationCatToLinearCat} immediately implies
the following result.
\begin{prop}
\label{prop:QuiverPresentationLocallyTrivialCat}Denote by $Q$ the
quiver of $\Bbbk E$ regarded as a subgraph of $E$. Let $R$ be a
relation on $Q^{\ast}$ such that $(Q,R)$ is a category presentation
for $E$. Then $(Q,R)$ is also a quiver presentation for $\mathbb{\Bbbk}E$.
\end{prop}
It is easy to describe the Cartan matrix of $\mathbb{\Bbbk}E$. It
is clear that $e_{k}\mathbb{\Bbbk}Ee_{r}$ is the $\Bbbk$-vector
space spanned by the hom-set $E(e_{r},e_{k})$. Therefore, we obtain
the following result.
\begin{lem}
\label{lem:CartanMatrixLocallyTrivial}The Cartan matrix of $\mathbb{\Bbbk}E$
is the $n\times n$ matrix whose $(i,j)$ entry is the number of morphisms
in the hom-set $E(e_{j},e_{i})$.
\end{lem}
\begin{rem}
Note that if we order the columns in an ascending order with respect
to the partial order defined in \defref{EICatAreDirected} then the
Cartan matrix of $\mathbb{\Bbbk}E$ is an upper unitriangular matrix.
\end{rem}

\section{\label{sec:RepresentationTheoryOfMonoidsOfPartialFunctions}Representation
theory of some order-related monoids of partial functions}

Recall that $\PO_{n}$ is the monoid of all order-preserving partial
functions, $\PF_{n}$ is the monoid of all order-decreasing partial
functions and $\PC_{n}$ is the monoid of all order-preserving and
order-decreasing partial functions. The goal of this paper is to describe
certain properties of the algebras $\Bbbk\PO_{n}$, $\Bbbk\PF_{n}$
and $\mathbb{\Bbbk}\PC_{n}$. For any $n\in\mathbb{N}$ we define
$[n]=\{1,\ldots,n\}$ and it will be convenient to set $[0]=\varnothing$.
Denote by $\E_{n}$ the category defined as follows. The objects of
$\E_{n}$ are all subsets of $[n]$ and given two subsets $A,B\subseteq[n]$
the hom-set $\E_{n}(A,B)$ consists of all onto functions with domain
$A$ and range $B$. We denote by $\EO_{n}$ the subcategory of $\E_{n}$
with the same set of objects but the hom-set $\EO_{n}(A,B)$ consists
only of order-preserving functions (with domain $A$ and range $B$).
Similarly, we denote by $\EF_{n}$ ($\EC_{n}$) the subcategory whose
morphisms are order-decreasing functions (respectively, order-preserving
\emph{and} order-decreasing functions). In \cite[Section 5]{Stein2016}
the author has proved an isomorphism of algebras 
\[
\mathbb{\Bbbk}\PO_{n}\simeq\mathbb{\Bbbk}\EO_{n},\quad\Bbbk\PF_{n}\simeq\mathbb{\Bbbk}\EF_{n},\quad\mathbb{\Bbbk}\PC_{n}\simeq\mathbb{\Bbbk}\EC_{n}.
\]

Using this isomorphism we were able to describe the quiver of these
algebras \cite[Propositions 5.2, 5.5 and 5.8]{Stein2016}. In this
section we will continue the study of these algebras. It is easy to
see that $\EO_{n}$, $\EF_{n}$, $\EC_{n}$ are all locally trivial
so we can apply the results of \secref{AlgebrasOfLocallyTrivialCat}.
In particular, we obtain a quiver presentation for each of these algebras
using \propref{QuiverPresentationLocallyTrivialCat}. For the sake
of comparison, we mention that monoid presentations of $\PO_{n}$,
$\PF_{n}$ and $\PC_{n}$ can be found in \cite[Section 3]{Fernandes2002},
\cite[Section 2]{Umar2018} and \cite[Section 6]{Umar2018} respectively. 
\begin{rem*}
One can consider also the injective version of this monoids. We denote
by $\IS_{n}$ the monoid of all injective partial functions on an
$n$-element set. We define 
\[
\IO_{n}=\PO_{n}\cap\IS_{n},\quad\IF_{n}=\PF_{n}\cap\IS_{n},\quad\IC_{n}=\PC_{n}\cap\IS_{n}.
\]
It is well known that $\IO_{n}$ is an $\Hc$-trivial inverse monoid
(see \cite{Fernandes1997}) so its algebra over any field is semisimple
and the invariants considered in this paper are trivial.  Description
of properties of $\Bbbk\IC_{n}$ and $\Bbbk\IF_{n}$ can be obtained
in a similar way to what is done for $\Bbbk\PC_{n}$ and $\Bbbk\PF_{n}$
in \Subsecref{The_Partial_Catalan} and \Subsecref{Order_Decreasing_Partial_Functio}
but we will not work out the details in this paper. 
\end{rem*}

\subsection{Order-preserving partial functions}

In this section we will study the representation theory of $\PO_{n}$
using the category $\EO_{n}$ defined above. Note that two objects
$A$ and $B$ in $\EO_{n}$ are isomorphic if and only if $|A|=|B|$
and in this case there is only one order-preserving onto function
from $A$ to $B$. In particular, all the endomorphism monoids are
trivial so $\EO_{n}$ is indeed locally trivial.

\paragraph{Loewy series and Loewy length}

We define a ``defect'' function, 
\[
\df:\EO_{n}^{1}\to\mathbb{N}\cup\{0\}
\]
by 
\[
\df(f)=|A|-|B|
\]
 for $f\in\EO_{n}(A,B)$. Note that $\EO_{n}$ is \emph{graded} with
respect to the function ``$\text{defect}$'', that is, 
\[
\df(gf)=\df(g)+\df(f)
\]
wherever the composition $g\cdot f$ is defined. It is clear that
$\df(f)=0$ if and only if $f$ is an isomorphism so \propref{RadOfEICat}
implies that
\[
\Rad\mathbb{\Bbbk}\EO_{n}=\SP\{f\in(\EO_{n})^{1}\mid\df(f)\geq1\}.
\]

The next proposition gives an explicit basis for all the terms of
the descending Loewy series of $\Bbbk\EO_{n}$. A similar observation
for the category of all epimorphisms is given in \cite[Lemma 4.3]{Stein2016}.
\begin{lem}
\label{lem:LoweySeriesPO}$\Rad^{k}\mathbb{\Bbbk}\EO_{n}=\SP\{f\in(\EO_{n})^{1}\mid\df(f)\geq k\}.$
\end{lem}
\begin{proof}
Since $\Rad\mathbb{\Bbbk}\EO_{n}$ is spanned by elements $f$ satisfying
$\df(f)\geq1$ it is clear that 
\[
\Rad^{k}\mathbb{\Bbbk}\EO_{n}\subseteq\SP\{f\in(\EO_{n})^{1}\mid\df(f)\geq k\}.
\]
The other inclusion is easily proved by induction. We have already
seen that the case $k=1$ is true. Assume that the statement is true
for $k-1$ and take a morphism $f\in\EO_{n}(A,B)$ such that $\df(f)=|A|-|B|\geq k$.
Now take some $b\in B$ such that $f^{-1}(b)$ contains more then
one element. Clearly, $B$ cannot be all $[n]$. Without loss of generality,
take $r\in[n]\backslash B$ such that $b\leq l<r$ implies $l\in B$
 (if no such $r$ exists, then there must be $r\in[n]\backslash B$
such that $r<l\leq b$ implies $l\in B$ and the proof is similar).
Denote by $a$ the maximal element of $f^{-1}(b)$ and define two
functions $g:A\to B\cup\{r\}$ and $h:B\cup\{r\}\to B$ by 
\[
g(x)=\begin{cases}
f(x)+1 & a\leq x\:\text{and}\:f(x)<r\\
f(x) & \text{otherwise }
\end{cases},\quad h(x)=\begin{cases}
x-1 & b+1\leq x\leq r\\
x & \text{otherwise.}
\end{cases}
\]
It is easy to verify that $g$ and $h$ are well defined order-preserving
onto functions and $hg=f$. By the induction hypothesis $g\in\Rad^{k-1}\mathbb{\Bbbk}\EO_{n}$
and $h\in\Rad\mathbb{\Bbbk}\EO_{n}$ so we are done.
\end{proof}
\begin{cor}
The Loewy length of $\mathbb{\Bbbk}\EO_{n}$ and hence of $\mathbb{\Bbbk}\PO_{n}$
is $n$.
\end{cor}
We now give an explicit formula for the dimensions of the terms of
the Loewy series. We state the following fact as a separate lemma
for later use.
\begin{lem}[{\cite[Part of Lemma 4.1]{Umar2004b}}]
\label{lem:CountOrderPreserving}Let $A$ and $B$ be subsets of
$[n]$ such that $m=|A|\geq|B|=l$. Then the number of onto order-preserving
functions from $A$ to $B$ is ${m-1 \choose l-1}$.
\end{lem}
\begin{proof}
Let $f:A\to B$ be an order-preserving function. Recall that $\ker f$
is the equivalence relation defined on $A$ by 
\[
a_{1}(\ker f)a_{2}\iff f(a_{1})=f(a_{2})
\]
The claim follows by observing that $\ker f$ divides $A$ into $l$
``convex'' subsets and there are ${m-1 \choose l-1}$ ways to choose
``barriers'' between these subsets.
\end{proof}
\begin{rem}
Note that \lemref{CountOrderPreserving} is true also for the case
where $l=0$ (i.e., $B$ is the empty set).
\end{rem}
\begin{lem}
$\dim\Rad^{k}\mathbb{\Bbbk}\PO_{n}={\displaystyle \sum_{m=k+1}^{n}\sum_{l=1}^{m-k}{n \choose m}{n \choose l}}{m-1 \choose l-1}$.
\end{lem}
\begin{proof}
By \lemref{LoweySeriesPO} we need to count all onto order-preserving
functions $\mbox{\ensuremath{f:A\to B}}$ (where $A,B\subseteq[n]$)
such that $|A|-|B|\geq k$. There are ${n \choose m}$ ways to choose
a domain set $A$ of size $m$ and ${n \choose l}$ ways to choose
an image set $B$ of size $l$. By \lemref{CountOrderPreserving}
there are ${m-1 \choose l-1}$ onto order-preserving functions from
$A$ to $B$. Now all that is left is to sum up all the possible sizes
for the domain and image.
\end{proof}
We now want to discuss some invariants of $\Bbbk\EO_{n}\simeq\mathbb{\Bbbk}\PO_{n}$
which are preserved by Morita equivalence. As explained in \secref{AlgebrasOfLocallyTrivialCat},
the algebra of the skeleton of $\EO_{n}$ is Morita equivalent to
$\mathbb{\Bbbk}\EO_{n}$ so we can pass to the skeleton of $\EO_{n}$,
that is, the full subcategory obtained from $\EO_{n}$ by choosing
one object of every isomorphism class. We will denote this skeleton
by $\SEO_{n}$. As mentioned above, two objects $A$ and $B$ of $\EO_{n}$
are isomorphic if and only if $|A|=|B|$. So we can regard $\SEO_{n}$
as the following category. The objects of $\SEO_{n}$ are the sets
$[k]$ for $0\leq k\leq n$ and the morphisms are all total order-preserving
epimorphisms between them. Note that $\Bbbk\SEO_{n}$ is a split basic
algebra by \corref{AlgebrasOfLocallyTrivialCatAreBasic}. 

\paragraph*{Blocks}

If $M$ is a monoid with zero, it is well known (see \cite[Remark 5.3]{Steinberg2015})
that its algebra can be decomposed into: 
\[
\mathbb{\Bbbk}M\simeq\mathbb{\Bbbk}\{0\}\times\Bbbk_{0}M\simeq\mathbb{\Bbbk}\times\Bbbk_{0}M.
\]
It is clear that $\SEO_{n}$ has two connected components (with $[0]$
being an isolated vertex). Therefore, \lemref{ConnectedComponentsOfLocallyTrivialCatAlg}
implies that the algebra $\mathbb{\mathbb{\Bbbk}\PO}_{n}$ also has
precisely two blocks. Therefore, we obtain the following corollary.
\begin{lem}
\label{lem:DecompositionOfCPTnToConnectedComponents}The decomposition
of $\Bbbk\PO_{n}$ into a direct product of connected algebras is

\[
\mathbb{\Bbbk}\PO_{n}\simeq\mathbb{\Bbbk}\times\mathbb{\Bbbk}_{0}\PO_{n}.
\]
\end{lem}

\paragraph*{Cartan matrix}

The category $\SEO_{n}$ has $n+1$ objects so by \lemref{CartanMatrixLocallyTrivial}
the Cartan matrix is an $(n+1)\times(n+1)$ upper unitriangular matrix
according to a natural ordering. The $(i,j)$ entry is the number
of arrows from $[j-1]$ to $[i-1]$ so by \lemref{CountOrderPreserving}
we obtain the following result.
\begin{lem}
\label{lem:CartanMatrixOfPOPascalMatrix}The Cartan matrix of $\mathbb{\Bbbk}\PO_{n}$
and $\mathbb{\Bbbk}\EO_{n}$ is an $(n+1)\times(n+1)$ upper unitriangular
matrix whose $(i,j)$ entry is ${j-2 \choose i-2}$ for $j\geq i$.
\end{lem}
\begin{rem}
The upper triangular $n\times n$ matrix whose $(i,j)$ entry for
$\mbox{\ensuremath{j\geq i}}$ is ${j-1 \choose i-1}$ is called the
upper-triangular Pascal matrix of size $n$. Therefore, \lemref{CartanMatrixOfPOPascalMatrix}
says that the Cartan matrix of $\mathbb{\Bbbk}\PO_{n}$ can be viewed
as an extended upper-triangular Pascal matrix where $0\leq i\leq j\leq n$.
\end{rem}

\paragraph*{Quiver presentation}

Now we will describe the quiver presentation of $\mathbb{\Bbbk}\SEO_{n}$.
By \lemref{QuiverOfCatTrivialMonoids} in order to describe the quiver
we need only to identify the irreducible morphisms. 
\begin{lem}[{\cite[Lemma 5.1]{Stein2016}}]
A morphism $m\in\SEO_{n}([k],[r])$ is irreducible if and only if
$k=r+1$.
\end{lem}
By \lemref{CountOrderPreserving} there are precisely $k$ order preserving
onto functions from $[k+1]$ to $[k]$ so we deduce the following
corollary which is \cite[Proposition 5.2]{Stein2016}.
\begin{cor}
\label{cor:QuiverOfPO_n}The vertex set of the quiver of $\mathbb{\mathbb{\Bbbk}\PO}_{n}$
is $\{[0],\ldots,[n]\}$. There are $k$ arrows from $[k+1]$ to $[k]$,
for $k=0,\ldots,n-1$, and no other arrows.
\end{cor}
By \propref{QuiverPresentationLocallyTrivialCat} we know that we
just need to find a presentation for $\SEO_{n}$ (where the generators
are the morphisms of the quiver $Q$). Denote by $\SEO_{n}^{\bullet}$
the category obtained from $\SEO_{n}$ by removing the isolated vertex
$\varnothing$. Clearly $\SEO_{n}$ and $\SEO_{n}^{\bullet}$ have
the same presentation relation. Instead of finding a presentation
for $\SEO_{n}$ directly, we will show that $\SEO_{n}^{\bullet}$
is isomorphic to another category whose presentation is well-known.
Recall that a strict order-preserving function $g:X\to Y$ for some
posets $(X,\leq)$ and $(Y,\leq)$ is a function for which $x_{1}<x_{2}$
implies $g(x_{1})<g(x_{2})$. Clearly any strict order-preserving
function between linear posets is injective. Denote by $\Delta$ the
category defined as follows: The vertices are $[k]$ for $0\leq k$
and the morphisms from $[r]$ to $[k]$ are all the strict order-preserving
functions between those sets (including a unique empty function $f:[0]\to[k]$
for every $k$). Clearly, $\Delta([r],[k])$ is an empty set unless
$r\leq k$. We denote by $\Delta_{n}$ the full subcategory of $\Delta$
whose objects are $[k]$ for $0\leq k\leq n$.

The category $\Delta$ is a fundamental tool in algebraic topology
(see \cite{Sanderson1971}). A $\Delta$-set (or a semi-simplicial
set) is a contravariant functor from $\Delta$ to the category of
sets. It is an object similar to a simplicial set but lacks its degeneracy
maps. Therefore, the opposite category $\Delta_{n}^{\op}$ captures
in some sense the idea of face maps between $k$-simplices up to dimension
$n$. Take some linearly ordered set $\{v_{0},\ldots,v_{k}\}$ of
$k+1$ vectors in $\mathbb{R}^{n}$ such that $v_{1}-v_{0},\ldots,v_{k}-v_{0}$
are all linearly independent. The convex set spanned by $\{v_{0},\ldots,v_{k}\}$
is called an $k$-simplex. A \emph{face} of an $k$-simplex is a convex
set spanned by some (ordered) subset of $\{v_{0},\ldots,v_{k}\}$
(where a \emph{$r$-dimensional face} is a face spanned by $r+1$
vectors). A \emph{face map }is a function that sends the simplex to
one of its faces. Note that a face map of a $k$-simplex to an $r$-dimensional
face is determined by a choice of $k-r$ elements to ``delete''.
One can view $\Delta_{n}^{\op}$ as the category of all face maps
between $k$-simplices for $0\leq k\leq n$. However, the category
$\Delta_{n}$ has several other notations in the literature. We claim:
\begin{prop}
\label{prop:IsomorphismOfOrderPreservingAndReverseStrictlyOrderPreserving}For
every natural $n$, the categories $\SEO_{n+1}^{\bullet}$ and $\Delta_{n}^{\op}$
are isomorphic.
\end{prop}
\begin{proof}
We will prove that $(\SEO_{n+1}^{\bullet})^{\op}$ is isomorphic to
$\Delta_{n}$. Denote by $\Gamma_{n}$ the category defined as follows:
The objects are $[k]$ for $1\leq k\leq n$ and the morphisms from
$[r]$ to $[k]$ are all the strict order-preserving functions $g:[r]\to[k]$
such that $g(1)=1$. Note that one can think of any function $g:[r]\to[k]$
as a function from $\{2,\ldots,r+1\}$ to $\{2,\ldots,k+1\}$ and
then add a fixed point $g(1)=1$. So it should be clear that $\Gamma_{n+1}$
is isomorphic to $\Delta_{n}$. To be explicit, one can consider a
functor $\mathcal{F}:\Delta_{n}\to\Gamma_{n+1}$ defined on objects
by $\mathcal{F}([k])=[k+1]$ and for every $g\in\Delta_{n}([r],[k])$
the morphism $\mathcal{F}(g):[r+1]\to[k+1]$ is defined by 
\[
\mathcal{F}(g)(i)=\begin{cases}
1 & i=1\\
g(i-1)+1 & \text{otherwise.}
\end{cases}
\]
It is easy to check that $\mathcal{F}$ is a functor and it has an
inverse $\mathcal{F}^{-1}:\Gamma_{n+1}\to\Delta_{n}$ given on morphisms
by 
\[
\mathcal{F}^{-1}(g)(i)=g(i+1)-1
\]
so $\Delta_{n}$ and $\Gamma_{n+1}$ are isomorphic categories. We
will now prove that $(\SEO_{n+1}^{\bullet})^{\op}$ is isomorphic
to $\Gamma_{n+1}$. We can think of morphisms of $(\SEO_{n+1}^{\bullet})^{\op}$
as being the inverses of the morphisms of $\SEO_{n+1}^{\bullet}$.
The inverse $f^{-1}:[k]\to[r]$ of some function $f:\in\SEO_{n+1}^{\bullet}([r],[k])$
is usually not a function, but we can work with it as a relation.
Define a functor $\mathcal{G}:$$(\SEO_{n+1}^{\bullet})^{\op}\to\Gamma_{n+1}$
in the following way. On objects $\mathcal{G}$ is the identity function
and on morphisms $\mathcal{G}$ is defined by 
\[
\mathcal{G}(f^{-1})(i)=\min f^{-1}(i).
\]
It is easy to observe that $\min f^{-1}(i)$ is a strict order-preserving
function and $\min f^{-1}(1)=1$. $\mathcal{G}$ is indeed a functor.
It is obvious that $\mathcal{G}$ sends identity morphisms to identity
morphisms. Moreover, for every two morphisms $f_{1}\in\SEO_{n+1}^{\bullet}([r],[k])$
and $f_{2}\in\SEO_{n+1}^{\bullet}([k],[m])$ we have that $\mathcal{G}(f_{2}^{-1}f_{1}^{-1})=\min f_{2}^{-1}(f_{1}^{-1}(i))$
and $\mathcal{G}(f_{2}^{-1})G(f_{1}^{-1})=\min f_{2}^{-1}(\min f_{1}^{-1}(i))$.
Since $f_{2}$ is order-preserving, it is clear that the minimal element
of $f_{2}^{-1}(f_{1}^{-1}(i))$ will be in the set $f_{2}^{-1}(\min f_{1}^{-1}(i))$
so 
\[
\mathcal{G}(f_{2}^{-1}f_{1}^{-1})=\min f_{2}^{-1}(f_{1}^{-1}(i))=\min f_{2}^{-1}(\min f_{1}^{-1}(i))=\mathcal{G}(f_{2}^{-1})\mathcal{G}(f_{1}^{-1})
\]
which proves that $\mathcal{G}$ is a functor. It is also easy to
see that $\mathcal{G}$ has an inverse. For a given $g\in\Gamma_{n+1}([k],[r])$
the inverse $\mathcal{G}^{-1}$ is given by 
\[
\mathcal{G}^{-1}(g)=f^{-1}
\]
where $f:[r]\to[k]$ is the order-preserving onto function that sends
$j\in[r]$ to the maximal $i$ such that $g(i)\leq j$. Again it is
easy to check that this is indeed an inverse and establish that $\mathcal{G}$
is an isomorphism. So $(\SEO_{n+1}^{\bullet})^{\op}$ is isomorphic
to $\Gamma_{n+1}$ and hence to $\Delta_{n}$ as required.
\end{proof}
In order to give a presentation, we will need some notation for the
arrows in the quiver. Clearly every $f\in\SEO_{n+1}([k+1],[k])$ (for
$k\geq1)$ is determined by choosing one pair of successive numbers
that will be sent to the same image. Hence, we can denote the arrows
in the quiver by $d_{i}^{k}$ for $1\leq k\leq n$ and $1\leq i\leq k$.
The arrow $d_{i}^{k}$ corresponds to the unique order preserving
onto function $f:[k+1]\to[k]$ such that $f(i)=f(i+1)$. By \propref{IsomorphismOfOrderPreservingAndReverseStrictlyOrderPreserving},
$d_{i}^{k}$ corresponds to some morphism in $(\Delta_{n})^{\op}$.
Following the explicit isomorphism given in \propref{IsomorphismOfOrderPreservingAndReverseStrictlyOrderPreserving},
it is easy to see that $d_{i}^{k}$ corresponds to the inverse of
the unique strict order preserving function from $[k-1]$ to $[k]$
such that $i$ is not in its image. Note that the generators $d_{i}^{k}$
correspond to $(k-1)$-dimensional face maps in the $k$-simplex interpretation.
The presentation for the category $(\Delta_{n})^{\op}$ according
to this set of generators is well known in algebraic topology. 
\begin{lem}[{\cite[Section VII.5 exercise 2a]{MacLane1998} or \cite[end of Section 2.3]{Friedman2012}}]
\label{lem:CategoryPresentationForDelta}The category $\Delta_{n}^{\op}$
and hence $\SEO_{n+1}$ is presented by the generating quiver $Q$
with morphisms $d_{i}^{k}$ as defined above and the relations
\[
d_{i}^{k-1}d_{j}^{k}=d_{j-1}^{k-1}d_{i}^{k}\quad(2\leq k\leq n,\quad1\leq i<j\leq k).
\]
\end{lem}
With \propref{QuiverPresentationLocallyTrivialCat} this immediately
implies:
\begin{cor}
Let $Q$ be the quiver of $\mathbb{\Bbbk}\SEO_{n}$ considered as
a subgraph as given in \corref{QuiverOfPO_n} and denote the arrows
of $Q$ by $d_{i}^{k}$ as above. A quiver presentation of $\mathbb{\Bbbk}\SEO_{n}$
and hence of $\mathbb{\Bbbk}\PO_{n}$ is given by 

\[
d_{i}^{k-1}d_{j}^{k}=d_{j-1}^{k-1}d_{i}^{k}\quad(2\leq k\leq n-1,\quad1\leq i<j\leq k)
\]
\end{cor}
\begin{rem}
Since the range and domain objects can usually be understood from
the context, the convention in the literature is to omit the superscripts.
The above relations are then written as: 
\[
d_{i}d_{j}=d_{j-1}d_{i}\quad(i<j)
\]
\end{rem}

\subsection{\label{subsec:The_Partial_Catalan}The partial Catalan monoid}

In this section we will study the representation theory of the partial
Catalan monoid $\PC_{n}$ using the category $\EC_{n}$. Recall that
the functions in $\PC_{n}$ are both order-preserving and order-decreasing.
For every set $A\subseteq[n]$, it is clear that the identity function
$1_{A}$ is the only order-decreasing function with domain and image
being $A$ so $\EC_{n}$ is indeed a locally trivial category. Moreover,
if $A,B\subseteq[n]$ for $A\neq B$ then at least one of the hom-sets
$\EC_{n}(A,B)$ or $\EC_{n}(B,A)$ is empty so the objects $A$ and
$B$ are not isomorphic hence $\EC_{n}$ is skeletal. We remark that
there is also a skeletal locally trivial category (in fact, a poset)
whose algebra is isomorphic to the algebra of the Catalan monoid $\C_{n}$,
i.e., the monoid of all \emph{total }order-preserving and order-decreasing
functions (see \cite{Margolis2018}). 

\paragraph{Blocks}

For every non-empty $A\subseteq[n]$ there exists a morphism in $\EC_{n}$
with domain $A$ and image $\{1\}$. Therefore, it is clear that the
category $\EC_{n}$ has two connected components (with $\varnothing$
being an isolated vertex). Since $\PC_{n}$ is a monoid with zero
we can use the same argument as in \lemref{DecompositionOfCPTnToConnectedComponents}
to obtain the following.
\begin{lem}
The decomposition of $\mathbb{\Bbbk}\PC_{n}$ into a direct product
of connected algebras is

\[
\Bbbk\PC_{n}\simeq\mathbb{\Bbbk}\times\Bbbk_{0}\PC_{n}.
\]
\end{lem}

\paragraph{Quiver presentation}

Since $\EC_{n}$ is a skeletal locally trivial category, its quiver
is the subgraph of all irreducible morphisms. So we just need to identify
the irreducible morphisms in order to find the quiver. This was done
in \cite{Stein2016} with the following result.
\begin{lem}[{\cite[Lemma 5.7]{Stein2016}}]
A morphism $f\in\EC_{n}(A,B)$ is irreducible if and only if there
exists $j\in A$ such that $f(i)=i$ for any $i\in A\backslash\{j\}$
and $f(j)=j-1$. 
\end{lem}
\begin{cor}
The vertices in the quiver of $\mathbb{\Bbbk}\PC_{n}$ and $\Bbbk\EC_{n}$
are in one-to-one correspondence with subsets of $[n]$. For $A,B\subseteq[n]$,
the arrows from $A$ to $B$ are in one-to-one correspondence with
onto functions $f:A\to B$ for which there exists $j\in A$ such that
$f(i)=i$ for $i\in A\backslash\{j\}$ and $f(j)=j-1$. 
\end{cor}
In this subsection we denote by $Q$ the quiver of $\Bbbk\EC_{n}$.
We now want to describe the quiver presentation of $\Bbbk\PC_{n}$
using \propref{QuiverPresentationLocallyTrivialCat}. We clearly need
some way to index the morphisms of $Q$. We denote by $d_{i}^{A}$
the irreducible morphism whose domain is $A$ and $i\in A$ is its
unique element such that $d_{i}^{A}(i)=i-1$. Note that the range
of $d_{i}^{A}$ is $(A\cup\{i-1\})\backslash\{i\}$. For simplicity
we denote this set by $A_{i}$. 
\begin{lem}
\label{lem:RelationsHoldInEC_n}For every set $A\subseteq[n]$ the
relations
\begin{enumerate}[label=(PC\arabic*)]
\item \label{enu:PC1}$d_{i}^{A_{j}}d_{j}^{A}=d_{j}^{A_{i}}d_{i}^{A}\quad(j>i+1,\quad i,j\in A)$
\item \label{enu:PC2}$d_{i}^{A_{i+1}}d_{i+1}^{A}=d_{i}^{(A_{i})_{i+1}}d_{i+1}^{A_{i}}d_{i}^{A}\quad(i,i+1\in A)$
\end{enumerate}
hold in $\EC_{n}$.
\end{lem}
\begin{rem}
In order to simplify notation, we will drop the superscripts and remain
with the ``braid like'' relations
\begin{enumerate}[label=(PC\arabic*)]
\item $d_{i}d_{j}=d_{j}d_{i}\quad(j>i+1,\quad i,j\in A)$
\item $d_{i}d_{i+1}=d_{i}d_{i+1}d_{i}\quad(i,i+1\in A)$
\end{enumerate}
where the domain of every morphism should be understood from the context.
\end{rem}
\begin{proof}[Proof of \lemref{RelationsHoldInEC_n}]
This is a straightforward verification for every $k\in A$. For \ref{enu:PC1}
we note that 
\[
d_{i}d_{j}(k)=d_{j}d_{i}(k)=\begin{cases}
k & k\neq i,j\\
i-1 & k=i\\
j-1 & k=j.
\end{cases}
\]
For \ref{enu:PC2} we have that
\[
d_{i}d_{i+1}(k)=d_{i}d_{i+1}d_{i}(k)=\begin{cases}
k & k\neq i,i+1\\
i-1 & k=i,i+1.
\end{cases}
\]
\end{proof}
In this subsection we will denote the category relation defined in
\lemref{RelationsHoldInEC_n} by $R$. We will show that $(Q,R)$
is a quiver presentation for $\Bbbk\PC_{n}$. Let $\theta_{R}$ is
the category congruence generated by $R$.
\begin{lem}
\label{lem:InductiveStepEC_n}Let $f:A\to B$ be a non-identity morphism
of $\EC_{n}$ and let 
\[
f=g_{1}\cdots g_{r}
\]
 be some decomposition of $f$ into irreducible morphisms. Denote
by $i\in A$ the minimal element $x\in A$ such that $f(x)<x$. Then
$g_{1}\cdots g_{r}$ is $\theta_{R}$ equivalent to 
\[
g_{1}^{\prime}\cdots g_{r^{\prime}}^{\prime}d_{i}
\]
for some irreducible morphisms $g_{1}^{\prime},\ldots,g_{r^{\prime}}^{\prime}$.
\end{lem}
\begin{proof}
We prove this by induction on the domain of $f$ according to the
partial order $\leq_{\EC_{n}}$ defined on the objects of $\EC_{n}$
(see \defref{EICatAreDirected}). If $f$ is irreducible then there
is nothing to prove. Now, consider a morphism $f:A\to B$ and assume
we have already proved the claim for every morphism with domain $X$
for $A<_{\EC_{n}}X$. If $g_{r}=d_{i}$ then there is nothing to prove.
Otherwise $g_{r}=d_{j}$ for some $i<j$. Define 
\[
h=g_{1}\cdots g_{r-1}
\]
Clearly the domain of $h$ is $A_{j}=(A\cup\{j-1\})\backslash\{j\}$
and $A<A_{j}$. Note that $i\in A_{j}$ and it must be the minimal
element $x\in A_{j}$ such that $h(x)<x$ so by the induction assumption
this decomposition is $\theta_{R}$ equivalent to
\[
g_{1}^{\prime}\cdots g_{l}^{\prime}d_{i}
\]
and therefore $g_{1}\cdots g_{r}$ is $\theta_{R}$ equivalent to
\[
g_{1}^{\prime}\cdots g_{l}^{\prime}d_{i}d_{j}.
\]

If $j>i+1$ then \ref{enu:PC1} implies that we can swap the two rightmost
morphisms and obtain
\[
g_{1}^{\prime}\cdots g_{l}^{\prime}d_{i}d_{j}=g_{1}^{\prime}\cdots g_{l}^{\prime}d_{j}d_{i}.
\]
If $j=i+1$ we can use \ref{enu:PC2} and get
\[
g_{1}^{\prime}\cdots g_{l}^{\prime}d_{i}d_{i+1}=g_{1}^{\prime}\cdots g_{l}^{\prime}d_{i}d_{i+1}d_{i}.
\]
In any case we get a new decomposition (which might be of different
length) 
\[
g_{1}^{\prime}\cdots g_{r^{\prime}}^{\prime}d_{i}
\]
as required.
\end{proof}
\begin{prop}
The tuple $(Q,R)$ is a category presentation for $\EC_{n}$.
\end{prop}
\begin{proof}
In view of \lemref{RelationsHoldInEC_n}, it is left to show that
these relations are enough. In other words, if $f$ is a morphism
of $\EC_{n}$ with two different decompositions into irreducible morphisms
\begin{align*}
f & =g_{1}\cdots g_{r}\\
f & =h_{1}\cdots h_{l}
\end{align*}
we need to prove that these decompositions are $\theta_{R}$ equivalent.
We will prove this by induction on the domain of $f$ according to
the partial order $\leq_{\EC_{n}}$. If $f$ is irreducible there
is nothing to prove. Now, consider a morphism $f:A\to B$ and assume
we have already proved the claim for every morphism with domain $X$
for $A<_{\EC_{n}}X$. Take $i$ to be the minimal element $x\in A$
such that $f(x)<x$ (such an element exists if $f$ is not an isomorphism).
By \lemref{InductiveStepEC_n} we know that $g_{1}\cdots g_{r}$ and
$h_{1}\cdots h_{l}$ are $\theta_{R}$ equivalent to $g_{1}^{\prime}\cdots g_{r^{\prime}}^{\prime}d_{i}$
and $h_{1}^{\prime}\cdots h_{l^{\prime}}^{\prime}d_{i}$ respectively.
Now, it is clear that the domain of both $g_{1}^{\prime}\cdots g_{r\prime}^{\prime}$
and $h_{1}^{\prime}\cdots h_{l^{\prime}}^{\prime}$ is $A_{i}=(A\cup\{i-1\})\backslash\{i\}$.
For every $k\in A_{i}$, if $k\neq i-1$ then 
\[
g_{1}^{\prime}\cdots g_{r^{\prime}}^{\prime}(k)=g_{1}^{\prime}\cdots g_{r^{\prime}}^{\prime}d_{i}(k)=f(k)=h_{1}^{\prime}\cdots h_{l^{\prime}}^{\prime}d_{i}(k)=h_{1}^{\prime}\cdots h_{l^{\prime}}^{\prime}(k)
\]
and if $k=i-1$ then 
\[
g_{1}^{\prime}\cdots g_{r^{\prime}}^{\prime}(i-1)=g_{1}^{\prime}\cdots g_{r\prime}^{\prime}d_{i}(i)=f(i)=h_{1}^{\prime}\cdots h_{l^{\prime}}^{\prime}d_{i}(i)=h_{1}^{\prime}\cdots h_{l^{\prime}}^{\prime}(i-1).
\]
Therefore, $g_{1}^{\prime}\cdots g_{r^{\prime}}^{\prime}$ and $h_{1}^{\prime}\cdots h_{l\prime}^{\prime}$
present the same function. Note that $\mbox{\ensuremath{A<_{\EC_{n}}A_{i}}}$
so by the inductive assumption, they are $\theta_{R}$ equivalent.
Hence $g_{1}^{\prime}\cdots g_{r^{\prime}}^{\prime}d_{i}$ and $h_{1}^{\prime}\cdots h_{l^{\prime}}^{\prime}d_{i}$
are also $\theta_{R}$ equivalent and this finishes the proof.
\end{proof}
In conclusion, we have the following.
\begin{thm}
Let $Q$ be the quiver of $\Bbbk\PC_{n}\simeq\mathbb{\Bbbk}\EC_{n}$
and denote the arrows of $Q$ by $d_{i}^{k}$ as above. A quiver presentation
of these algebras is given by the relations
\begin{align*}
d_{i}^{A_{j}}d_{j}^{A} & =d_{j}^{A_{i}}d_{i}^{A}\quad(j>i+1,\quad i,j\in A)\\
d_{i}^{A_{i+1}}d_{i+1}^{A} & =d_{i}^{(A_{i})_{i+1}}d_{i+1}^{A_{i}}d_{i}^{A}\quad(i,i+1\in A)
\end{align*}

for every $A\subseteq[n]$.
\end{thm}
\begin{rem}
Note that there is some similarity between the quiver presentation
of $\Bbbk\PC_{n}$ and the monoid presentation of the Catalan monoid
$\C_{n}$ by ``Kiselman relations'' (see \cite{Grensing2014}).
\end{rem}

\paragraph{Cartan matrix}

The category $\EC_{n}$ has $2^{n}$ objects and therefore, $\mathbb{\Bbbk}\EC_{n}$
has $2^{n}$ irreducible representations, which are naturally indexed
by subsets of $[n]$. Given $A,B\subseteq[n]$, \lemref{CartanMatrixLocallyTrivial}
implies that the $(B,A)$ entry of the Cartan matrix is the number
of (total) onto order-preserving and order-decreasing functions $f:A\to B$.
We would like to give some method to enumerate this number. We denote
by $\C(A,B)$ the set of all order-preserving and order-decreasing
functions $f:A\to B$ and by $\EC(A,B)$ the onto functions of $\C(A,B)$.
We will start by giving a way to count the elements of $\EC([n],B)$.
By the inclusion exclusion principle on the poset of subsets of $[n]$
(see \cite[Section 2.1]{Stanley1997}), it is clear that 
\[
|\EC([n],B)|=\sum_{X\subseteq B}(-1)^{|B|-|X|}|\C([n],X)|.
\]
Therefore, it is enough to count the elements of $\C([n],B)$ in order
to get an expression for $|\EC([n],B)|$. It is well known that elements
of $\C([n],[n])$ are in one-to-one correspondence with (North-East)
lattice paths from $(1,1)$ to $\mbox{\ensuremath{(n+1,n+1)}}$ that
remain below the line $y=x$. For details see \cite{Higgins1993}
or the introduction of \cite{Grensing2014} (a correspondence between
$\PC_{n}$ and another type of lattice paths can be found in \cite{Umar2004}).
Order-preserving and order-decreasing functions with image contained
in $B$ correspond to lattice paths whose horizontal steps, i.e. steps
of the form $(i,j)$ to $(i+1,j)$, satisfy $j\in B$. It will be
convenient to use $n$-tuples instead of lattice paths. Every lattice
path from $(1,1)$ to $(n+1,n+1)$ can be identified with an $n$-tuple
$(p_{1},\ldots p_{n})$ where $p_{i}$ is the $y$ coordinate of the
$(i,j)\to(i+1,j)$ step. In the other direction any $n$-tuple $P=(p_{1},\ldots,p_{n})$
which is non decreasing and its elements satisfy $1\leq p_{i}\leq n+1$
corresponds to some lattice path from $(1,1)$ to $(n+1,n+1)$. Therefore
we can represent such paths with $n$-tuples. We say that a lattice
path $P=(p_{1},\ldots,p_{n})$ is below a lattice path $T=(t_{1},\ldots,t_{n})$
and write $P\leq T$ if $p_{i}\leq t_{i}$ for every $i$. This clearly
defines a partial order on lattice paths. It is clear that elements
of $\C([n],B)$ correspond to lattice paths $P=(p_{1},\ldots,p_{n})$
such that $p_{i}\in B$ for every $i$ and $P\leq(1,2,\ldots,n)$
(since they are below $y=x$).
\begin{defn}
For every $B\subseteq[n]$, denote by $\bar{B}=\{\bar{b}_{1},\ldots,\bar{b}_{n}\}$
the lattice path such that $\overline{b_{1}}=1$ and for $i>1$ we
have
\[
\overline{b_{i}}=\begin{cases}
\overline{b_{i-1}} & i\notin B\\
\overline{b_{i-1}}+1 & i\in B.
\end{cases}
\]
In other words, the ascends of $\bar{B}$ are in positions $i\in B$.
\end{defn}
\begin{example}
\label{exa:BarB1}If $n=8$ and $B=\{1,3,4,5,8\}$ then 
\[
\bar{B}=(1,1,2,3,4,4,4,5).
\]
\end{example}
\begin{lem}
There is a one-to-one correspondence between $\C([n],B)$ and lattice
paths (from $(1,1)$ to $(n+1,n+1)$) $P$ such that $P\leq\bar{B}$.
\end{lem}
\begin{proof}
We have already seen that there is a one-to-one correspondence between
$\C([n],B)$ and the set $\mathcal{P}$ of all lattice paths $P=(p_{1},\ldots,p_{n})$
such that $p_{i}\in B$ for every $i$ and $P\leq(1,2,\ldots,n)$.
Denote by $P_{B}$ a lattice path whose $i$-th element is the maximal
$b\in B$ such that $b\leq i$. It is easy to see that $P_{B}\in\mathcal{P}$
is a maximum element. Therefore $\mathcal{P}$ is the set of lattice
paths $P=(p_{1},\ldots,p_{n})$ such that $p_{i}\in B$ for every
$i$ and $P\leq P_{B}$. Now the result follows by renaming the name
of elements. More precisely, if $B=\{b_{1},\ldots,b_{k}\}$ we can
define a partial permutation $\sigma_{B}:B\to[k]$ by $\sigma_{B}(b_{i})=i$.
It is clear that $\sigma_{B}$ preserves order and that $\sigma_{B}(P_{B})=\bar{B}$
(where $\sigma_{B}(P_{B})$ means acting by $\sigma_{B}$ componentwise).
Therefore there is a one-to-one correspondence between $\mathcal{P}$
and $\sigma_{B}(\mathcal{P})$ which is the set of lattice paths $P$
such that $P\leq\bar{B}$.
\end{proof}
\begin{example}
If $B=\{1,3,4,5,8\}$ as in \exaref{BarB1} then 
\[
P_{B}=\left(\begin{array}{cccccccc}
1 & 1 & 3 & 4 & 5 & 5 & 5 & 8\end{array}\right)
\]
and 
\[
\sigma_{B}=\left(\begin{array}{ccccc}
1 & 3 & 4 & 5 & 8\\
1 & 2 & 3 & 4 & 5
\end{array}\right)
\]
so indeed $\sigma_{B}(P_{B})=\bar{B}.$
\end{example}

\begin{thm}[{Part of \cite[Theorem 10.7.1]{Krattenthaler2015}}]
\label{thm:EnumeratePaths}Given a lattice path $X=(x_{1},\ldots,x_{n})$
from $(1,1)$ to $(n+1,n+1)$, define a matrix $M_{X}$ by 
\[
\left[M_{X}\right]_{i,j}={x_{i} \choose j-i+1}.
\]
The number of lattice paths $P=(p_{1},\ldots p_{n})$ from $(1,1)$
to $(n+1,n+1)$ which satisfy $P\leq X$ is the determinant of $M_{X}$.
\end{thm}
\begin{cor}
\label{cor:SizeOfOrderPreservingOrderIncreasing}The size of $\C([n],B)$
is the determinant of the matrix $M_{\bar{B}}$ where $\bar{B}=(\bar{b}_{1},\ldots,\bar{b}_{n})$
as defined above.
\end{cor}
As mentioned above, the inclusion exclusion principle on the poset
of subsets of $[n]$ immediately gives us the following corollary:
\begin{cor}
\label{cor:EnumerationCartanEC_n}The number of order-preserving and
order-decreasing onto functions $f:[n]\to B$ is given by
\[
\sum_{X\subseteq B}(-1)^{|B|-|X|}|M_{\bar{X}}|.
\]
\end{cor}
Now we want to count the elements in the set $\EC(A,B)$ where the
domain $A$ is not $[n]$ but some subset. If $|A|=m$ we will show
that the number of order-preserving and order-decreasing onto functions
$f:A\to B$ is the number of such functions $f:[m]\to B^{\prime}$
for some appropriate choice of $B^{\prime}$ so it can be also be
enumerated by \corref{EnumerationCartanEC_n}. 
\begin{lem}
Let $A,B\subseteq[n]$ such that $|A|=m$. There exists $A^{\prime}$
such that $B\subseteq A^{\prime}$ and $|\EC(A,B)|=|\EC(A^{\prime},B)|$.
\end{lem}
\begin{proof}
Assume $B=\{b_{1},\ldots,b_{k}\}$, ordered by the standard order.
We will build $A^{\prime}$ from $A$ in $k$ steps. At the first
step we take the minimal element $a\in A$ such that $b_{1}\leq a$
and define $A_{1}=(A\backslash\{a\})\cup\{b_{1}\}$. It should be
clear that $|\EC(A_{1},B)|=|\EC(A,B)|$. We repeat this process with
the other elements of $B$. In the $i$-th step we have already obtained
a set $A_{i-1}$ such $|\EC(A_{i-1},B)|=|\EC(A,B)|$ such that $b_{1},\ldots,b_{i-1}\in A_{i-1}$.
Now take the minimal element $a\in A$ such that $b_{i}\leq a$ and
define $A_{i}=(A_{i-1}\backslash\{a\})\cup\{b_{i}\}$. Now it is not
difficult to see that $|\EC(A_{i},B)|=|\EC(A_{i-1},B)|$. Formally
we can define a bijection $\tau_{a}:A_{i-1}\to A_{i}$ which is the
identity on $A_{i-1}\backslash\{a\}$ and $\tau_{a}(a)=b_{i}$. Now
the function $\Phi:\EC(A_{i},B)\to\EC(A_{i-1},B)$ defined by $\Phi(f)=f\tau_{a}$
is clearly a bijection between the two sets. Finally we define $A^{\prime}=A_{k}$
and it is clear that $B\subseteq A^{\prime}$ and $|\EC(A^{\prime},B)|=|\EC(A,B)|$. 
\end{proof}
It is now left to count the set $\EC(A^{\prime},B)$. Assume $A^{\prime}=\{a_{1},\ldots,a_{m}\}$
ordered by the standard order. Define $\sigma_{A^{\prime}}$ to be
the partial bijection $\sigma_{A^{\prime}}(a_{i})=i$ and denote $\sigma_{A^{\prime}}(B)=\{\sigma_{A^{\prime}}(b_{1}),\ldots,\sigma_{A^{\prime}}(b_{k})\}$.
\begin{lem}
The following equality holds 
\[
|\EC(A^{\prime},B)|=|\EC([m],\sigma_{A^{\prime}}(B))|
\]
\end{lem}
\begin{proof}
Since $\sigma_{A^{\prime}}$ is a (partial) permutation which preserves
order, we can think of applying it as just renaming the elements of
the sets so the claim is obvious.
\end{proof}
In conclusion, we have displayed a method to enumerate the set $\EC(A,B)$
of all onto order-preserving and order-decreasing functions $f:A\to B$
which is the $(B,A)$ entry of the Cartan matrix.

\paragraph{Loewy length}

Recall that $\Rad\Bbbk\EC_{n}$ is spanned by all the non-invertible
morphisms of $\EC_{n}$, i.e. all the non identity morphisms.
\begin{lem}
\label{lem:LoweyLengthLemma}Let $f\in\EC_{n}(A,B)$ be a non-identity
morphism. $f$ is the composition of at most ${n \choose 2}$ non-identity
morphisms.
\end{lem}
\begin{proof}
In this proof it will be more convenient to assume that the objects
of $\EC_{n}$ are all the subsets of $\{0,\ldots,n-1\}$. For every
set $A\subseteq\{0,\ldots,n-1\}$ define $S(A)$ to be the sum of
its elements
\[
S(A)=\sum_{a\in A}a.
\]
Since $f:A\to B$ is onto and order-decreasing, we must have that
$S(B)<S(A)$. Since $S(\{0,\ldots,n-1\})={n \choose 2}$ it is clear
that a morphism cannot be written as a composition of more than ${n \choose 2}$
non-identity morphism.
\end{proof}
\begin{prop}
\label{prop:LoweyLengthOfPC_n}The Loewy length of $\Bbbk\EC_{n}$
and hence of $\mathbb{\Bbbk}\PC_{n}$ is ${n \choose 2}+1$.
\end{prop}
\begin{proof}
By \lemref{LoweyLengthLemma} it is clear that $\Rad^{k}\mathbb{\Bbbk}\EC_{n}=0$
where $k={n \choose 2}+1$. It is only left to prove that $\Rad^{k}\mathbb{\Bbbk}\EC_{n}\neq0$
where $k={n \choose 2}$. Recall that we have denoted by $d_{i}^{A}$
the irreducible morphism whose domain is $A$ and $i\in A$ is its
unique element such that $d_{i}^{A}(i)=i-1$. For $2\leq i\leq n$
define $f_{i}:\{1,i,i+1,\ldots,n\}\to\{1,i+1,\ldots,n\}$ to be the
morphism which is the identity on all elements except $f_{i}(i)=1$.
It is clear that $f_{i}$ can be written as a composition of $i-1$
morphisms 
\[
f_{i}=d_{2}\cdots d_{i-1}d_{i}
\]
where we have dropped the superscripts because they can be understood
from the context. Now, it is easy to see that the constant function
$\mathbf{1}:[n]\to\{1\}$ can be written as the following composition
\[
\boldsymbol{1}=f_{n}\cdots f_{3}f_{2}
\]
and therefore we have found a morphism which can be written as a composition
of ${n \choose 2}$ morphisms and we are done.
\end{proof}

\subsection{\label{subsec:Order_Decreasing_Partial_Functio}Order-decreasing
partial functions}

In this section we will study the representation theory of the monoid
$\PF_{n}$ of all order-decreasing partial functions using the category
$\EF_{n}$. We remark that $\PF_{n}$ is isomorphic to the monoid
of all order-decreasing \emph{total} functions on $n+1$ elements.
This fact was first observed in \cite[Corollary 2.4.3]{UmarThesis1992}.
Another proof due to the referee of \cite{Stein2016} can be found
in \cite[Lemma 5.3]{Stein2016}. We remark that a monoid is $\Lc$-trivial
(that is, any two distinct elements generate different left ideals)
if and only if it is isomorphic to a submonoid of $\PF_{n}$ for some
$n\in\mathbb{N}$ \cite[Chapter 4, Theorem 3.6]{Pin1986}. For every
set $A\subseteq[n]$, it is clear that the identity function $1_{A}$
is the only order-decreasing function with domain and image being
$A$ so $\EF_{n}$ is indeed a locally trivial category. Moreover,
if $A,B\subseteq[n]$ for $A\neq B$ then at least one of the hom-sets
$\EF_{n}(A,B)$ or $\EF_{n}(B,A)$ is empty so the objects $A$ and
$B$ are not isomorphic hence $\EF_{n}$ is skeletal. 

\paragraph*{Blocks}

It is clear that for every non-empty $A\subseteq[n]$ there exists
a constant order-decreasing function $f:A\to\{1\}$. Therefore the
category $\EF_{n}$ has precisely two connected components with the
$\varnothing$ object being isolated. Since $\PF_{n}$ is a monoid
with zero we can use the same argument as in \lemref{DecompositionOfCPTnToConnectedComponents}
to obtain the following.
\begin{lem}
The decomposition of $\Bbbk\PF_{n}$ into a direct product of connected
algebras is

\[
\Bbbk\PF_{n}\simeq\mathbb{\Bbbk}\times\Bbbk_{0}\PF_{n}.
\]
\end{lem}

\paragraph{Cartan matrix}

The category $\EF_{n}$ has $2^{n}$ objects and therefore, $\mathbb{\Bbbk}\EF_{n}$
has $2^{n}$ irreducible representations, which are naturally indexed
by subsets of $[n]$. Given $B\subseteq[n]$ and $i\in[n]$ we denote
\[
B_{\leq i}=\{b\in B\mid b\leq i\}
\]

\begin{lem}
\label{lem:NumberOfWeaklyDecreasingFunctions}The number of order-decreasing
(total, but not necessarily onto) functions $f:A\to B$ is 
\[
\prod_{i\in A}|B_{\leq i}|
\]
\end{lem}
\begin{proof}
The image of every $i\in A$ could be any element in $B$ which is
smaller than $i$ hence there are $|B_{\leq i}|$ options. The choice
of image of any two elements of $A$ is independent so we just take
the product of the number of options.
\end{proof}
\begin{lem}
The Cartan matrix of $\mathbb{\Bbbk}\EF_{n}$ is a $2^{n}\times2^{n}$
matrix. Given $A,B\subseteq[n]$ the $(B,A)$ entry of the Cartan
matrix is 
\[
\sum_{X\subseteq B}(-1)^{|B|-|X|}\prod_{i\in A}|X{}_{\leq i}|
\]
\end{lem}
\begin{proof}
The $(B,A)$ entry of the Cartan matrix is the number of (total) onto
order-decreasing functions $f:A\to B$ by \lemref{CartanMatrixLocallyTrivial}.
The claim follows immediately by the inclusion-exclusion principle
on the poset of subsets of $[n]$ (see \cite[Section 2.1]{Stanley1997})
and \lemref{NumberOfWeaklyDecreasingFunctions}.
\end{proof}

\paragraph{Quiver presentation}

Describing a quiver presentation for $\mathbb{\Bbbk}\PF_{n}$ is similar
to the case of $\mathbb{\Bbbk}\PC_{n}$ but a bit more complicated.
Again, $\EF_{n}$ is a skeletal locally trivial category, so its quiver
is the subgraph of all irreducible morphisms. In order to describe
the irreducible morphisms we will use the following notation. Let
$A\subseteq[n]$ and let $i,j\in[n]$ be two distinct elements. We
will write $i\vartriangleleft_{A}j$ if $i<j$ and $i<x\leq j$ implies
that $x\in A$. In other words, if all the elements between $i$ and
$j$ (including $j$) are in $A$. 
\begin{lem}[{\cite[Lemma 5.4]{Stein2016}}]
\label{lem:IrreduciblesInEF_n} A morphism $f\in\EF_{n}(A,B)$ is
irreducible if and only if there exists $j\in A$ such that $f(i)=i$
for any $i\in A\backslash\{j\}$ and $f(j)\vartriangleleft_{A}j$. 
\end{lem}
\begin{cor}
The vertices in the quiver of $\mathbb{\Bbbk}\PF_{n}$ and $\mathbb{\Bbbk}\EF_{n}$
are in one-to-one correspondence with subsets of $[n]$. For $A,B\subseteq[n]$,
the arrows from $A$ to $B$ are in one-to-one correspondence with
onto functions $f:A\to B$ for which there exists $j\in A$ such that
$f(i)=i$ for $i\in A\backslash\{j\}$ and $f(j)\vartriangleleft_{A}j$. 
\end{cor}
Now denote by $Q$ the quiver of $\Bbbk\EF_{n}$. We now want to describe
the quiver presentation of $\mathbb{\Bbbk}\PF_{n}$ using \propref{QuiverPresentationLocallyTrivialCat}.
We index the morphisms of $Q$ in the following way. We denote by
$d_{i,j}^{A}$ the irreducible morphism whose domain is $A$ and $j\in A$
is its unique element such that $d_{i,j}^{A}(j)=i\neq j$. Clearly,
using this notation implies that $i\vartriangleleft_{A}j$. Note that
the range of $d_{i,j}^{A}$ is $(A\cup\{i\})\backslash\{j\}$. For
simplicity we denote this set by $A_{i,j}$. 
\begin{lem}
\label{lem:RelationsHoldInEF_n}Let $A\subseteq[n]$ and assume $j<t$.
The relations
\begin{enumerate}[label=(PF\arabic*)]
\item \label{enu:PF1}$d_{i,j}^{A_{s,t}}d_{s,t}^{A}=d_{s,t}^{A_{i,j}}d_{i,j}^{A}\quad(s>j,\quad t,j\in A)$
\item \label{enu:PF2}$d_{i,j}^{A_{s,t}}d_{s,t}^{A}=d_{i,j}^{A_{i,t}}d_{j,t}^{A_{i,j}}d_{i,j}^{A}\quad(s=j,\quad t,j\in A)$
\item \label{enu:PF3}$d_{i,j}^{A_{s,t}}d_{s,t}^{A}=d_{s,j}^{A_{i,t}}d_{j,t}^{A_{i,j}}d_{i,j}^{A}\quad(i<s<j,\quad s,t,j\in A)$
\item \label{enu:PF4}$d_{i,j}^{A_{s,t}}d_{s,t}^{A}=d_{s,j}^{A_{i,t}}d_{j,t}^{A_{i,j}}d_{i,j}^{A}\quad(s\leq i,\quad t,j\in A)$
\item \label{enu:PF5}$d_{i,j}^{A_{s,t}}d_{s,t}^{A}=d_{s,j}^{A_{i,t}}d_{j,t}^{A_{i,j}}d_{i,s}^{A_{s,j}}d_{s,j}^{A}\quad(i<s<j,\quad t,j\in A,\quad s\notin A)$
\item \label{enu:PF6}$d_{i,j}^{(A_{i,t})_{s,i}}d_{s,i}^{A_{i,t}}d_{i,t}^{A}=d_{s,j}^{A_{i,t}}d_{j,t}^{A_{i,j}}d_{i,j}^{A}\quad(s<i<j<t,\quad t,j\in A,\quad i\notin A)$
\end{enumerate}
hold in $\EF_{n}$.
\end{lem}
\begin{rem}
In order to simplify notation, we will drop the superscripts and remain
with the relations
\begin{enumerate}[label=(PF\arabic*)]
\item $d_{i,j}d_{s,t}=d_{s,t}d_{i,j}\quad(s>j,\quad t,j\in A)$
\item $d_{i,j}d_{s,t}=d_{i,j}d_{j,t}d_{i,j}\quad(s=j,\quad t,j\in A)$
\item $d_{i,j}d_{s,t}=d_{s,j}d_{j,t}d_{i,j}\quad(i<s<j,\quad s,t,j\in A)$
\item $d_{i,j}d_{s,t}=d_{s,j}d_{j,t}d_{i,j}\quad(s\leq i,\quad t,j\in A)$
\item $d_{i,j}d_{s,t}=d_{s,j}d_{j,t}d_{i,s}d_{s,j}\quad(i<s<j,\quad t,j\in A,\quad s\notin A)$
\item $d_{i,j}d_{s,i}d_{i,t}=d_{s,j}d_{j,t}d_{i,j}\quad(s<i<j<t,\quad t,j\in A,\quad i\notin A)$
\end{enumerate}
where the domain of every morphism should be understood from the context.
\end{rem}
\begin{rem}
Note that relations \ref{enu:PF1}-\ref{enu:PF5} cover all the cases
of a term $d_{i,j}d_{s,t}$ satisfying $j<t$.
\end{rem}
\begin{proof}[Proof of \lemref{RelationsHoldInEF_n}]
This is a routine matter to check that all the compositions in \lemref{RelationsHoldInEF_n}
are well defined. Now, it is a straightforward verification to check
equality. Choose some $k\in A$. For \ref{enu:PF1},\ref{enu:PF3},\ref{enu:PF4}
and \ref{enu:PF5} we note that

\[
d_{i,j}d_{s,t}(k)=d_{s,t}d_{i,j}(k)=\begin{cases}
k & k\neq t,j\\
i & k=j\\
s & k=t
\end{cases}
\]
\[
d_{i,j}d_{s,t}(k)=d_{s,j}d_{j,t}d_{i,j}(k)=\begin{cases}
k & k\neq t,j\\
i & k=j\\
s & k=t
\end{cases}
\]
\[
d_{i,j}d_{s,t}(k)=d_{s,j}d_{j,t}d_{i,s}d_{s,j}(k)=\begin{cases}
k & k\neq t,j\\
i & k=j\\
s & k=t.
\end{cases}
\]
For \ref{enu:PF2} we have that 
\[
d_{i,j}d_{s,t}(k)=d_{i,j}d_{j,t}d_{i,j}(k)=\begin{cases}
k & k\neq t,j\\
i & k=j,t
\end{cases}
\]
and for \ref{enu:PF6} we obtain (note that $k\neq i$ as $i\notin A$)
\[
d_{i,j}d_{s,i}d_{i,t}(k)=d_{s,j}d_{j,t}d_{i,j}(k)=\begin{cases}
k & k\neq t,j\\
i & k=j\\
s & k=t.
\end{cases}
\]
\end{proof}
In this subsection we will denote the category relation defined in
\lemref{RelationsHoldInEF_n} by $R$. We will show that $(Q,R)$
is a quiver presentation for $\Bbbk\PF_{n}$.
\begin{lem}
\label{lem:InductiveStepEF_n}Let $f:A\to B$ be a non-identity morphism
of $\EF_{n}$ and let 
\[
f=g_{1}\cdots g_{r}
\]
 be some decomposition of $f$ into irreducible morphisms. Let $j\in A$
be the minimal element $x\in A$ such that $f(x)<x$. Then $g_{1}\cdots g_{r}$
is $\theta_{R}$ equivalent to 
\[
g_{1}^{\prime}\cdots g_{r^{\prime}}^{\prime}d_{i,j}
\]
for some irreducible morphisms $g_{1}^{\prime},\ldots,g_{r^{\prime}}^{\prime}$
and some $i\vartriangleleft_{A}j$ (where $\theta_{R}$ is the category
congruence generated by $R$).
\end{lem}
\begin{proof}
We prove this by induction on the domain of $f$ according to the
partial order $\leq_{\EF_{n}}$ defined on the objects of $\EF_{n}$
(see \defref{EICatAreDirected}). If $f$ is irreducible then there
is nothing to prove. Now, consider a morphism $f:A\to B$ and assume
we have already proved the claim for every morphism with domain $X$
for $A<_{\EC_{n}}X$. If $g_{r}=d_{i,j}$ then we are done. Otherwise
$g_{r}=d_{i_{1},t}$ for some $t>j$. Define $h_{1}=g_{1}\cdots g_{r-1}$.
It is clear that the domain of $h$ is $A_{i_{1},t}$ and that $j\in A_{i_{1},t}$
and $A<_{\EC_{n}}A_{i_{1},t}$. 
\begin{casenv}
\item Assume that $j$ is the minimal element $x\in A_{i_{1},t}$ such that
$h(x)<x$. In this case the induction assumption implies that $g_{1}\cdots g_{r-1}$
is $\theta_{R}$ equivalent to $g_{1}^{\prime}\cdots g_{l}^{\prime}d_{q,j}$
and therefore $g_{1}\cdots g_{r}$ is $\theta_{R}$ equivalent to
$g_{1}^{\prime}\cdots g_{l}^{\prime}d_{q,j}d_{i_{1},t}$. Recall that
$j<t$ so by one of the relations \ref{enu:PF1}-\ref{enu:PF5} we
can ``push'' the $d_{q,j}$ term to the rightmost position and obtain
$g_{1}^{\prime}\cdots g_{r^{\prime}}^{\prime}d_{i,j}$ as required.
This finishes this case.
\item It might be the case that $j$ is not the minimal element $x\in\dom h_{1}$
such that $h_{1}(x)<x$. The only other possibility for such minimal
element is $i_{1}$. In this case, the induction assumption implies
that $g_{1}\cdots g_{r-1}$ is $\theta_{R}$ equivalent to $g_{1}^{(1)}\cdots g_{l_{1}}^{(1)}d_{i_{2},i_{1}}$
so $g_{1}\cdots g_{r-1}d_{i_{1},t}$ is $\theta_{R}$ equivalent to
$g_{1}^{(1)}\cdots g_{l_{1}}^{(1)}d_{i_{2},i_{1}}d_{i_{1},t}$. Now
denote $h_{2}=g_{1}^{(1)}\cdots g_{l_{1}}^{(1)}$ by a similar argument,
if $j$ is not the minimal $x\in\dom h_{2}$ such that $h_{2}(x)<x$
then it must be $i_{2}$ so we obtain a $\theta_{R}$ equivalence
with $g_{1}^{(2)}\cdots g_{l_{2}}^{(2)}d_{i_{3},i_{2}}d_{i_{2},i_{1}}d_{i_{1},t}$.
This process must terminate at some point. Eventually we obtain a
decomposition 
\[
g_{1}^{(n-1)}\cdots g_{l_{n-1}}^{(n-1)}d_{i_{n},i_{n-1}}\cdots d_{i_{2},i_{1}}d_{i_{1},t}
\]
where $j$ is the minimal element $x$ in the domain of $h_{n}=g_{1}^{(n-1)}\cdots g_{l_{n-1}}^{(n-1)}$
such that $h_{n}(x)<x$. Denote the domain of $h_{n}$ by $A_{n}$.
By the induction assumption this decomposition is $\theta_{R}$ equivalent
to 
\[
g_{1}^{(n)}\cdots g_{l_{n}}^{(n)}d_{q,j}d_{i_{n},i_{n-1}}\cdots d_{i_{2},i_{1}}d_{i_{1},t}.
\]
Now, we know that $i_{1}<j$. However, it cannot be the case that
$q<i_{1}<j$ because in this case $q\ntriangleleft_{A_{n}}j$ in contrary
to the existence of the $d_{q,j}$ term. Therefore, $i_{1}\leq q$
and $i_{n},\ldots,i_{2}<q$ so we can use \ref{enu:PF1} to swap terms
and obtain $g_{1}^{(n)}\cdots g_{l_{n-1}}^{(n)}d_{i_{n},i_{n-1}}\cdots d_{q,j}d_{i_{2},i_{1}}d_{i_{1},t}$
Now, if $i_{1}<q$ we can again swap terms with \ref{enu:PF1} to
get 
\[
g_{1}^{(n)}\cdots g_{l_{n-1}}^{(n)}d_{i_{n},i_{n-1}}\cdots d_{i_{2},i_{1}}d_{q,j}d_{i_{1},t}
\]
and finish the proof with another use of \ref{enu:PF1}-\ref{enu:PF5}.
If $i_{1}=q$, the decomposition is of the form 
\[
g_{1}^{(n)}\cdots g_{l_{n-1}}^{(n)}d_{i_{n},i_{n-1}}\cdots d_{i_{1},j}d_{i_{2},i_{1}}d_{i_{1},t}.
\]
By the minimality of $j$, we must have that $i_{1}\notin A$ so we
can use \ref{enu:PF6} to obtain 
\[
g_{1}^{(n)}\cdots g_{l_{n-1}}^{(n)}d_{i_{n},i_{n-1}}\cdots d_{i_{2},j}d_{j,t}d_{i_{1},j}
\]
which finishes this case and the proof.
\end{casenv}
\end{proof}
\begin{prop}
The tuple $(Q,R)$ is a category presentation for $\EF_{n}$.
\end{prop}
\begin{proof}
In view of \lemref{RelationsHoldInEF_n}, it is left to show that
these relations are enough. In order words, if $f$ is a morphism
of $\EF_{n}$ with two different decompositions into irreducible morphisms
\begin{align*}
f & =g_{1}\cdots g_{r}\\
f & =h_{1}\cdots h_{l}
\end{align*}
we need to prove that these decompositions are $\theta_{R}$ equivalent.
We will prove this by induction on the domain of $f$ according to
the partial order $\leq_{\EF_{n}}$ defined on the objects of $\EF_{n}$.
If $f$ is irreducible then there is nothing to prove. Now, consider
a morphism $f:A\to B$ and assume we have already proved the claim
for every morphism with domain $X$ for $A<_{\EC_{n}}X$. Take $j$
to be the minimal element $x\in A$ such that $f(x)<x$ (such an element
exists if $f$ is not an isomorphism). By \lemref{InductiveStepEF_n}
we know that $g_{1}\cdots g_{r}$ and $h_{1}\cdots h_{l}$ are $\theta_{R}$
equivalent to $g_{1}^{\prime}\cdots g_{r^{\prime}}^{\prime}d_{i_{1},j}$
and $h_{1}^{\prime}\cdots h_{l^{\prime}}^{\prime}d_{i_{2},j}$ respectively.
We first claim that $i_{1}=i_{2}$. Assume without loss of generality
that $i_{1}<i_{2}$. This implies that $f(j)<i_{2}$. Moreover, $i_{1}\vartriangleleft_{A}j$
so $i_{2}\in A$ and therefore 
\[
f(i_{2})=h_{1}^{\prime}\cdots h_{l^{\prime}}^{\prime}d_{i_{2},j}(i_{2})=h_{1}^{\prime}\cdots h_{l^{\prime}}^{\prime}(i_{2})=h_{1}^{\prime}\cdots h_{l^{\prime}}^{\prime}d_{i_{2},j}(j)=f(j)<i_{2}.
\]
This contradicts the minimality of $j$ and therefore $i_{1}=i_{2}=i$.
So $g_{1}\cdots g_{r}$ and $h_{1}\cdots h_{l}$ are $\theta_{R}$
equivalent to $g_{1}^{\prime}\cdots g_{r^{\prime}}^{\prime}d_{i,j}$
and $h_{1}^{\prime}\cdots h_{l^{\prime}}^{\prime}d_{i,j}$ respectively.
Now, it is clear that the domain of both $g_{1}^{\prime}\cdots g_{r^{\prime}}^{\prime}$
and $h_{1}^{\prime}\cdots h_{l^{\prime}}^{\prime}$ is $A_{i,j}=(A\cup\{i\})\backslash\{j\}$.
For every $k\in A_{i,j}$, if $k\neq i$ then 
\[
g_{1}^{\prime}\cdots g_{r^{\prime}}^{\prime}(k)=g_{1}^{\prime}\cdots g_{r^{\prime}}^{\prime}d_{i,j}(k)=f(k)=h_{1}^{\prime}\cdots h_{l^{\prime}}^{\prime}d_{i,j}(k)=h_{1}^{\prime}\cdots h_{l^{\prime}}^{\prime}(k)
\]
and if $k=i$ then 
\[
g_{1}^{\prime}\cdots g_{r^{\prime}}^{\prime}(i)=g_{1}^{\prime}\cdots g_{r^{\prime}}^{\prime}d_{i,j}(j)=f(j)=h_{1}^{\prime}\cdots h_{l^{\prime}}^{\prime}d_{i,j}(j)=h_{1}^{\prime}\cdots h_{l^{\prime}}^{\prime}(i).
\]
Therefore, $g_{1}^{\prime}\cdots g_{r^{\prime}}^{\prime}$ and $h_{1}^{\prime}\cdots h_{l^{\prime}}^{\prime}$
present the same function.  Note that $A<A_{i,j}$ so by the inductive
assumption, they are $\theta_{R}$ equivalent. Hence $g_{1}^{\prime}\cdots g_{r^{\prime}}^{\prime}d_{i,j}$
and $h_{1}^{\prime}\cdots h_{l^{\prime}}^{\prime}d_{i,j}$ are also
$\theta_{R}$ equivalent and this finishes the proof.
\end{proof}
In conclusion, we have the following.
\begin{thm}
Let $Q$ be the quiver of $\Bbbk\PF_{n}\simeq\mathbb{\Bbbk}\EF_{n}$.
A quiver presentation of these algebras is given by the relations

\begin{align*}
d_{i,j}^{A_{s,t}}d_{s,t}^{A} & =d_{s,t}^{A_{i,j}}d_{i,j}^{A}\quad(s>j,\quad t,j\in A)\\
d_{i,j}^{A_{s,t}}d_{s,t}^{A} & =d_{i,j}^{A_{i,t}}d_{j,t}^{A_{i,j}}d_{i,j}^{A}\quad(s=j,\quad t,j\in A)\\
d_{i,j}^{A_{s,t}}d_{s,t}^{A} & =d_{s,j}^{A_{i,t}}d_{j,t}^{A_{i,j}}d_{i,j}^{A}\quad(i<s<j,\quad s,t,j\in A)\\
d_{i,j}^{A_{s,t}}d_{s,t}^{A} & =d_{s,j}^{A_{i,t}}d_{j,t}^{A_{i,j}}d_{i,j}^{A}\quad(s\leq i,\quad t,j\in A)\\
d_{i,j}^{A_{s,t}}d_{s,t}^{A} & =d_{s,j}^{A_{i,t}}d_{j,t}^{A_{i,j}}d_{i,s}^{A_{s,j}}d_{s,j}^{A}\quad(i<s<j,\quad t,j\in A,\quad s\notin A)\\
d_{i,j}^{(A_{i,t})_{s,i}}d_{s,i}^{A_{i,t}}d_{i,t}^{A} & =d_{s,j}^{A_{i,t}}d_{j,t}^{A_{i,j}}d_{i,j}^{A}\quad(s<i<j<t,\quad t,j\in A,\quad i\notin A)
\end{align*}

for $j<t$ and every $A\subseteq[n]$.
\end{thm}

\paragraph{Loewy length}
\begin{prop}
The Loewy length of $\Bbbk\EF_{n}$ and hence of $\Bbbk\PF_{n}$ is
${n \choose 2}+1$.
\end{prop}
\begin{proof}
The proof is similar to the case of $\mathbb{\Bbbk}\EC_{n}$. Note
that $\EC_{n}$ is a subcategory of $\EF_{n}$. An identical argument
of \lemref{LoweyLengthLemma} proves that no morphism can be written
as a composition of ${n \choose 2}+1$ non-identity elements and \propref{LoweyLengthOfPC_n}
proves that there exists a morphisms which is a composition of ${n \choose 2}$
non-identity morphisms.
\end{proof}
\textbf{Acknowledgments: }The author would like to thank Prof. Stuart
Margolis for pointing out that \lemref{PresentationCatToLinearCat}
can be proved by a categorical argument. The author is grateful to
the referees for their valuable comments and suggestions.

\bibliographystyle{plain}
\bibliography{library}

\begin{thebibliography}{10}

\bibitem{Assem2006}
Ibrahim Assem, Daniel Simson, and Andrzej Skowro{\'n}ski.
\newblock {\em Elements of the representation theory of associative algebras.
  {V}ol. 1}, volume~65 of {\em London Mathematical Society Student Texts}.
\newblock Cambridge University Press, Cambridge, 2006.
\newblock Techniques of representation theory.

\bibitem{Auslander1997}
Maurice Auslander, Idun Reiten, and Sverre~O. Smal\o.
\newblock {\em Representation theory of {A}rtin algebras}, volume~36 of {\em
  Cambridge Studies in Advanced Mathematics}.
\newblock Cambridge University Press, Cambridge, 1997.
\newblock Corrected reprint of the 1995 original.

\bibitem{Fernandes1997}
Vitor~H. Fernandes.
\newblock Semigroups of order preserving mappings on a finite chain: a new
  class of divisors.
\newblock {\em Semigroup Forum}, 54(2):230--236, 1997.

\bibitem{Fernandes2002}
V\'{\i}tor~H. Fernandes.
\newblock Presentations for some monoids of partial transformations on a finite
  chain: a survey.
\newblock In {\em Semigroups, algorithms, automata and languages ({C}oimbra,
  2001)}, pages 363--378. World Sci. Publ., River Edge, NJ, 2002.

\bibitem{Friedman2012}
Greg Friedman.
\newblock Survey article: an elementary illustrated introduction to simplicial
  sets.
\newblock {\em Rocky Mountain J. Math.}, 42(2):353--423, 2012.

\bibitem{Ganyushkin2009b}
Olexandr Ganyushkin and Volodymyr Mazorchuk.
\newblock {\em Classical finite transformation semigroups}, volume~9 of {\em
  Algebra and Applications}.
\newblock Springer-Verlag London Ltd., London, 2009.
\newblock An introduction.

\bibitem{Grensing2014}
Anna-Louise Grensing and Volodymyr Mazorchuk.
\newblock Categorification of the {C}atalan monoid.
\newblock {\em Semigroup Forum}, 89(1):155--168, 2014.

\bibitem{Higgins1993}
Peter~M. Higgins.
\newblock Combinatorial results for semigroups of order-preserving mappings.
\newblock {\em Math. Proc. Cambridge Philos. Soc.}, 113(2):281--296, 1993.

\bibitem{Krattenthaler2015}
Christian Krattenthaler.
\newblock Lattice path enumeration.
\newblock In {\em Handbook of enumerative combinatorics}, Discrete Math. Appl.
  (Boca Raton), pages 589--678. CRC Press, Boca Raton, FL, 2015.

\bibitem{Umar2004}
A.~Laradji and A.~Umar.
\newblock Combinatorial results for semigroups of order-decreasing partial
  transformations.
\newblock {\em J. Integer Seq.}, 7(3):Article 04.3.8, 14, 2004.

\bibitem{Umar2004b}
A.~Laradji and A.~Umar.
\newblock Combinatorial results for semigroups of order-preserving partial
  transformations.
\newblock {\em J. Algebra}, 278(1):342--359, 2004.

\bibitem{Li2011}
Liping Li.
\newblock A characterization of finite {EI} categories with hereditary category
  algebras.
\newblock {\em J. Algebra}, 345:213--241, 2011.

\bibitem{Luck}
Wolfgang L\"{u}ck.
\newblock {\em Transformation groups and algebraic {$K$}-theory}, volume 1408
  of {\em Lecture Notes in Mathematics}.
\newblock Springer-Verlag, Berlin, 1989.
\newblock Mathematica Gottingensis.

\bibitem{MacLane1998}
Saunders Mac~Lane.
\newblock {\em Categories for the working mathematician}, volume~5 of {\em
  Graduate Texts in Mathematics}.
\newblock Springer-Verlag, New York, second edition, 1998.

\bibitem{Margolis2015}
Stuart Margolis, Franco Saliola, and Benjamin Steinberg.
\newblock Cell complexes, poset topology and the representation theory of
  algebras arising in algebraic combinatorics and discrete geometry.
\newblock {\em arXiv preprint arXiv:1508.05446}, 2015.

\bibitem{Margolis2012}
Stuart Margolis and Benjamin Steinberg.
\newblock Quivers of monoids with basic algebras.
\newblock {\em Compos. Math.}, 148(5):1516--1560, 2012.

\bibitem{Margolis2018}
Stuart Margolis and Benjamin Steinberg.
\newblock The algebra of the catalan monoid as an incidence algebra: A simple
  proof.
\newblock {\em arXiv preprint arXiv:1806.06531}, 2018.

\bibitem{Mitchell1972}
Barry Mitchell.
\newblock Rings with several objects.
\newblock {\em Advances in Math.}, 8:1--161, 1972.

\bibitem{Pin1986}
Jean~Eric Pin.
\newblock {\em Varieties of formal languages}.
\newblock Foundations of Computer Science. Plenum Publishing Corp., New York,
  1986.
\newblock With a preface by M.-P. Sch{\"u}tzenberger, Translated from the
  French by A. Howie.

\bibitem{Ringel2000}
C.~M. Ringel.
\newblock The representation type of the full transformation semigroup {$T_4$}.
\newblock {\em Semigroup Forum}, 61(3):429--434, 2000.

\bibitem{Sanderson1971}
Colin~Patrick Rourke and Brian~Joseph Sanderson.
\newblock {$\Delta$}-sets. {I}. {H}omotopy theory.
\newblock {\em Quart. J. Math. Oxford Ser. (2)}, 22:321--338, 1971.

\bibitem{Saliola2009}
Franco~V. Saliola.
\newblock The face semigroup algebra of a hyperplane arrangement.
\newblock {\em Canad. J. Math.}, 61(4):904--929, 2009.

\bibitem{Stanley1997}
Richard~P. Stanley.
\newblock {\em Enumerative combinatorics. {V}ol. 1}, volume~49 of {\em
  Cambridge Studies in Advanced Mathematics}.
\newblock Cambridge University Press, Cambridge, 1997.
\newblock With a foreword by Gian-Carlo Rota, Corrected reprint of the 1986
  original.

\bibitem{Stein2017erratum}
Itamar Stein.
\newblock Erratum to: Algebras of {E}hresmann semigroups and categories.
\newblock In {\em Semigroup Forum}, pages 1--5. Springer.

\bibitem{Stein2016}
Itamar Stein.
\newblock The representation theory of the monoid of all partial functions on a
  set and related monoids as {EI}-category algebras.
\newblock {\em J. Algebra}, 450:549--569, 2016.

\bibitem{Stein2017}
Itamar Stein.
\newblock Algebras of {E}hresmann semigroups and categories.
\newblock {\em Semigroup Forum}, 95(3):509--526, 2017.

\bibitem{Steinberg2015}
Benjamin Steinberg.
\newblock {\em The representation theory of finite monoids}.
\newblock 2015.

\bibitem{Tilson1987}
Bret Tilson.
\newblock Categories as algebra: an essential ingredient in the theory of
  monoids.
\newblock {\em J. Pure Appl. Algebra}, 48(1-2):83--198, 1987.

\bibitem{UmarThesis1992}
Abdullahi Umar.
\newblock {\em Semigroups of order-decreasing transformations}.
\newblock PhD thesis, University of St Andrews, 1992.

\bibitem{Umar2018}
Abdullahi Umar.
\newblock Presentations for subsemigroups of {$PD_n$}.
\newblock {\em Czechoslovak Mathematical Journal}, Oct 2018.

\bibitem{Wang2017}
Shoufeng Wang.
\newblock On algebras of {$P$}-{E}hresmann semigroups and their associate
  partial semigroups.
\newblock {\em Semigroup Forum}, 95(3):569--588, 2017.

\bibitem{Webb2007}
Peter Webb.
\newblock An introduction to the representations and cohomology of categories.
\newblock In {\em Group representation theory}, pages 149--173. EPFL Press,
  Lausanne, 2007.

\end{thebibliography}

\end{document}